\documentclass{amsart}
\usepackage{amsmath,amsthm,amsfonts,amssymb,latexsym,mathrsfs,graphicx}

\usepackage{hyperref}

\usepackage{enumerate}
\usepackage[shortlabels]{enumitem}
\usepackage{color}

\usepackage{tikz}
\usetikzlibrary{decorations.pathreplacing}

\usepackage{ytableau}
\newtheorem{theorem}{Theorem}[section]

\newtheorem{lemma}[theorem]{Lemma}
\newtheorem{remark}[theorem]{Remark}

\theoremstyle{definition}

\newtheorem*{ThmA}{Theorem A}
\newtheorem*{ThmB}{Theorem B}
\newtheorem*{ThmC}{Theorem C}
\newtheorem*{ThmD}{Theorem D}

\newcommand{\Irr}{{\mathrm {Irr}}}

\newcommand{\Aut}{{\mathrm {Aut}}}

\newcommand{\Syl}{\mathrm{Syl}}
\newcommand{\Van}{\mathrm{Van}}
\newcommand{\Nv}{\mathrm{Nv}}

\begin{document}
	
	\title{common non-zero  graphs of groups  with  few edges}
	
	\author[Z. Akhlaghi]{Zeinab Akhlaghi}
%	\author[K. Azizi]{Kamal Aziziheris}
	\author[S. Y. Madanha]{Sesuai Y. Madanha}
	\address{Corresponding author: Zeinab Akhlaghi, Faculty of Math. and Computer Sci., \newline Amirkabir University of Technology (Tehran Polytechnic), 15914 Tehran, Iran.}
		%\newline
	%	and \newline
	%	School of Mathematics,\newline
	%	Institute for Research in Fundamental Science(IPM)
	%	P.O. Box:19395-5746, Tehran, Iran.}
	\email{z\_akhlaghi@aut.ac.ir}
	
%	\address{Kamal Aziziheris, Department of Pure Mathematics,\newline  Faculty of Mathematical Sciences, University
%		of Tabriz, Tabriz, Iran\newline
%		and \newline
%		School of Mathematics,\newline
%		Institute for Research in Fundamental Science(IPM)
%		P.O. Box:19395-5746, Tehran, Iran.}
%	\email{azizi@tabrizu.ac.ir}
\address{Sesuai Y. Madanha, Department of Mathematics and Applied Mathematics,\newline University of Pretoria, Hatfield, Pretoria, 0028, South Africa}
\email{sesuai.madanha@up.ac.za}
%\thanks{This work is based upon research funded by Iran National Science Foundation (INSF) under project No  40505623.}
%\thanks{The first author also would like to thank the Department of Mathematics and Applied Mathematics at the University of Pretoria for their hospitality during her stay there in May and early June 2025.}
%\thanks{	The second author  is supported by  a grant from IPM (No. 1404200020).}


%\MSC[MSC 2010]{20C15, 20D10, 20D15}
\begin{abstract}
Let \( G \) be a finite group, and consider its common non-zero graph, which we  denote by  \( \Gamma_{nv}(G) \). In this graph, the vertices represent the non-linear irreducible characters of \( G \). There is an edge connecting two distinct vertices \( \chi_1 \) and \( \chi_2 \) if there is a vanishing element \( g \in G \) such that \( \chi_1(g) \chi_2(g) \neq 0 \). A finite group \( G \) is described as a strongly \( \mathcal{H}_1' \)-group if the set of vanishing elements of each of its non-linear irreducible character is equal to   the set of vanishing elements of \( G \).
In this paper, we prove Conjecture 1 from \cite{ourself}, which states that if \( \Gamma_{nv}(G) \) is null (has no edge), then \( G \) must either be a strongly \( \mathcal{H}_1' \)-group or a Frobenius group, whose Frobenius complement is isomorphic to \( Q_8 \).
Additionally, we show that if there are no triangles in \( \Gamma_{nv}(G) \), this implies that the group is solvable. We also look at the situation where the common non-zero graph of a finite group \( G \) contains only one edge. This leads us to conclude that \( G \) is either isomorphic to \( S_4 \); or that the single edge in \( \Gamma_{nv}(G) \) is of the form \( \{\theta, \zeta\theta\} \) for some \( \zeta \in \Irr(G/G') \) and \( \theta \in \Irr(G|G') \), $G$ is a $p$-nilpotent group for some prime $p$ and  for all characters \( \chi \in \Irr(G|G') - \{\theta, \zeta\theta\} \), we find \( \Van(\chi) = \Van(G) \). Lastly, our study also explores groups whose common non-zero graph forms a star shape, concluding that such configurations can only occur if \( G \) has exactly one non-linear irreducible character.
\end{abstract}

%\maketitle

\maketitle

\section{Introduction}

The study of characters of finite groups is a central theme in representation theory, with deep connections to the structure of the underlying group. One of the most fundamental results in this area, due to Burnside, states that every non-linear irreducible character must vanish at some group element. This simple yet profound observation has inspired a rich line of research investigating the relationship between the vanishing behavior of characters and the algebraic structure of the group. In particular, the set of elements where a character vanishes—and, more broadly, the collection of all vanishing elements of a group—carries substantial information about the group's structure.

In this paper, we work on  finite group \( G \) with  emphasis on its complex characters. We denote the set of irreducible characters of \( G \) as \( \Irr(G) \) and use \( \chi \) to refer to an  irreducible complex character. An element \( g \in G \) is  a zero of \( \chi \) when \( \chi(g) = 0 \). Additionally, we call \( g \) a vanishing element of \( G \) if there exists some \( \chi \in \Irr(G) \) that vanishes at \( g \). Note that this situation cannot occur when \( g \) is in the center denoted as \( \mathbf{Z}(G) \) of \( G \). Consequently, we refer to the conjugacy class of \( g \) as \( g^G \), labeling it a vanishing conjugacy class of \( G \), if $g$ is a vanishing element of $G$.

It is known that linear characters do not vanish at any element. However, as it is mentioned above that every non-linear irreducible character must vanish at some element. This important result was refined by Malle, Navarro, and Olsson \cite{MNO}, who showed that we can specifically select such an element to be of prime-power order. Additionally, it’s well-established that if \( G \) has an irreducible character that vanishes on all non-central elements, then \( G \) is solvable, and its structure is completely defined (see \cite{centraltype}). This result leads the authors of  \cite{ourself} to consider the structure of groups with a more expansive trait: those that have an irreducible character vanishing on all vanishing conjugacy classes.

 %there exists a subset \( C \subseteq \Irr(G) \) with \( |C| \leq k \), where every vanishing conjugacy class acts as a zero for at least one of the characters in \( C \). This prompts us to explore whether \( G \) is solvable because it qualifies as an \( \mathcal{H}'_1 \)-group.    The evidence leans toward a negative conclusion regarding this question. For example, when we examine the non-solvable Frobenius group \( G = C_{11}^2 \rtimes \SL_2(5) \), we find that it meets the requirements of a \( \mathcal{H}'_1 \)-group.  In \cite{gagula}, we can see more examples of non-solvable $\mathcal{H}_1'$-groups.  In a study from 1983 \cite{gagula}, Gagola outlined all groups with a character vanishing on all but two conjugacy classes. It was established that such groups possess a single unique minimal normal subgroup, with the two special classes being non-vanishing. Thus, these groups can be classified as \( \mathcal{H}'_1 \)-groups.
We denote the set of zeros or vanishing elements  of a character \( \chi \) by \( {\rm Van}(\chi) \). Similarly, we denote the set of vanishing elements of the group \( G \) by \( {\rm Van}(G) \). This can be articulated mathematically as follows: $$ {\rm Van}(G) = \bigcup_{\chi \in {\rm Irr}(G)} {\rm Van}(\chi). $$

Moreover, by $\Nv(\chi)$, where $\chi\in \Irr(G)$,  we mean the set of elements of $G$ that $\chi$ is not  vanishing on them and  $\Nv(G)=G-\Van(G)=\bigcap_{\chi\in \Irr(G)} \Nv(\chi)$. In the paper \cite{ourself}, a group \( G \) is identified as a strongly \( \mathcal{H}'_k \)-group if it meets the requirement that for every character \( \chi \) within \( \Irr(G|G') \), there are \( k-1 \) characters (not necessarily distinct) \( \chi^c_{i} \in \Irr(G/\ker\chi) \) for \( 1 \leq i \leq k-1 \), such that the union of the vanishing set of \( \chi \) and the vanishing sets of these \( k-1 \) characters equals the vanishing set of \( G \). It’s evident that if \( G \) qualifies as a strongly \(  \mathcal{H}'_k \)-group, then it also qualifies as a strongly  \( \mathcal{H}'_m \)-group for any \( m \geq k \). The authors also provide a classification of strongly \( \mathcal{H}'_1 \)-groups. Furthermore, they achieve a complete classification of non-solvable strongly \( \mathcal{H}'_2 \)-groups, revealing that there is only one family of such groups, specifically \( A \times A_7 \), where \( A \) is an abelian group.

%Thus, we can say that $G$ qualifies as an $\mathcal{H}'_1$-group if and only if ${\rm Van}(G) = {\rm Van}(\chi)$ for some character $\chi \in {\rm Irr}(G)$. In contrast, $G$ is identified as an $\mathcal{H}'_2$-group if ${\rm Van}(G)$ can be represented as the union of the vanishing sets associated with two distinct irreducible characters $\chi$ and $\psi$ from ${\rm Irr}(G)$. 
%For instance, according to \cite{Atlas}, we find that ${\rm Van}(S_{3}) = {\rm Van}(\chi)$, where the character $\chi$ is of degree 2 within ${\rm Irr}(S_{3})$. On the other hand, ${\rm Van}(A_{5})$ can be expressed as ${\rm Van}(\chi_{1}) \cup {\rm Van}(\chi_{2}) \cup {\rm Van}(\chi_{3})$, with $\chi_{1}, \chi_{2}, \chi_{3} \in {\rm Irr}(A_{5})$ having degrees 3, 4, and 5, respectively. This leads us to conclude that $S_{3}$ is classified as an $\mathcal{H}'_{1}$-group, while $A_{5}$ is recognized as an $\mathcal{H}'_{3}$-group, but it does not fall under the category of any $\mathcal{H}'_{i}$-group for $1 \leq i \leq 2$.
%For further illustrations of $\mathcal{H}_1'$-groups, one could examine the character table of groups such as $G \cong C_3 \rtimes D_4$, $M_9$, or $C_5^2 \rtimes Dic_3$. Among the almost simple groups, both $A_7$ and $S_7$ are classified as $\mathcal{H}_2'$-groups.

To gain a deeper insight into the characteristics of strongly $\mathcal{H}_k'$-groups, the authors in \cite{ourself} introduced the concept of the {\it common non-zero graph} associated with a group $G$, which is symbolized as $\Gamma_{nv}(G)$. In this graph, the vertices correspond to the non-linear irreducible characters of $G$. An edge connects two distinct irreducible characters, $\chi_1$ and $\chi_2$,  if there exists an element $g \in \Van(G)$ such that $\chi_1(g) \chi_2(g) \neq 0$. 
This concept parallels the common-zero graph explored in \cite{hung}. Notably, if $G$ is identified as a strongly $\mathcal{H}_1'$-group, then $\Gamma_{nv}(G)$ is a null graph (a graph with no edge). However, the converse is not necessarily valid. For example, the Frobenius group $C_3^2:Q_8$ exhibits a null graph but does not qualify as a strongly $\mathcal{H}_1'$-group. It is demonstrated in \cite[Theorem A]{ourself} that:

\begin{lemma}\label{thA}
	Let $G$ be  a finite group. Then $\Gamma_{nv}(G)$ is null if and only if one of the following occurs:
	
	1) $G$ is a strongly $\mathcal{H}_1'$-group; or

	2) $G$ has a normal subgroup $K$ such that $\Irr(G/K|G'K/K)=\{\theta\}$, $\Van(G)\cap G'K\not =\emptyset$  and  $\Van(\chi) = \Van(G)$ for all $\chi\in \Irr(G|G')-\{\theta\}$.
\end{lemma}

Additionally, in \cite[Theorem B]{ourself} the structure of strongly $\mathcal{H}_1'$-groups are determined as follows: 

\begin{lemma}\label{ThmB}
	Let $ G $ be a  finite non-abelian  group, where $G$ is a strongly $\mathcal{H}_1'$-group.   Then $ G $ is of Fitting height at most two and one of the following holds:
	\begin{itemize}
		\item[(a)]  $ G/\textbf{Z}(G) $ and $ G' $ are elementary abelian $ p $-groups for some prime $ p $ and $ G $ is of nilpotence class $ 2 $;
		\item[(b)] $ G/\textbf{F}(G) $ is cyclic and there exists a normal  subgroup $ R\subseteq \textbf{F}(G) $ such that $ G/R $ is a Frobenius group with an elementary abelian kernel and a cyclic Frobenius complement;
		%	\item[(c)]  $ G/\textbf{F}(G) $ is an elementary abelian $ p $-group for some prime $ p $ and there exists normal subgroup $ R\subseteq \textbf{F}(G) $ such that $ G/R $ is a $ VZ $-group.
	\end{itemize}

\end{lemma}

Moreover, to complete the classification  of groups whose common non-zero graph is null, in the same paper it is conjectured that (\cite[Conjecture 1]{ourself}): 
 Either $G$ is a strongly $\mathcal{H}_1'$-group, or $G$ is a Frobenius group whose Frobenius complements are isomorphic to  $Q_{8}$.
As one of our main results we give a   affirmative answer to this  conjecture.
  \begin{ThmA}
  Let $G$ be a finite non-abelian group. If  $\Gamma_{nv}(G)$ is null then either $G$ is a strongly $\mathcal{H}_1'$-group, or $ G $ is a Frobenius group whose Frobenius kernel is abelian  and the Frobenius complement is $ Q_{8} $. 
  \end{ThmA}

 Note that if $G$ is a strongly $\mathcal{H}_1'$-group then $\Gamma_{nv}(G)$ is null. Moreover If $G$ is a Frobenius group whose Ferobenius kernel is a $p$-group for some prime $p$ and the Frobenius complement is $Q_8$, then it is easy to see that $\Gamma_{nv}(G)$ is null. Thus according to the evidences we believe that the converse of Theorem A is valid.

  Based on Theorem A and Lemma \ref{ThmB}, we have identified all groups for which the common non-zero graphs are null. It makes sense for us to further investigate how a common non-zero graph influences the structure of a group by looking at finite groups that have few edges in their common non-zero graph. The null case represents the most restrictive condition, and the classification of such groups naturally raises the question: what happens when the graph is almost null? In particular, groups whose common non-zero graph has exactly one edge, is triangle-free, or forms a star represent the next layers of complexity. Understanding these cases provides insight into how the vanishing behavior of characters becomes increasingly complex as the graph gains more structure.
  
  We present the following:
  
\begin{ThmB}
	Let $G$ be a finite group whose common non-zero graph has exactly one edge $\{\theta_1, \theta_2\}$.  Then $\theta_1=\zeta\theta_2$ for some $\zeta\in \Irr(G/G')$ and  one of the following occurs:
	
	(1)   $\Van(\chi)=\Van(G)$ for  all $\chi\in \Irr(G|G')-\{\theta_1, \theta_2\}$. In particular $G$ is $p$-nilpotent for some prime $p$. 
	
	 (2)   $G\cong S_4$. 
\end{ThmB}
According to this theorem, the only group with a common non-zero graph that contains exactly one edge and three characters, whose vanishing sets do not coincide with the vanishing set of \( G \), is \( S_4 \). All other groups possess only two characters with this property, which are clearly the endpoints of the sole edge in the graph.
 It is noteworthy that there exist examples fulfilling Case 1 of Theorem B, which can be either nilpotent or non-nilpotent. For an example of a nilpotent group that satisfies Case 1, we can consider SmallGroup(32,6), a \( 2 \)-group. In fact all the $2$-groups satisfying Part (c) of  \cite[Lemma 2]{Ber96} are examples of Case 1  of Theorem B . On the other hand, for a non-nilpotent example, we can look at SmallGroup(108, 32) in GAP \cite{gap}, which also meets the requirements specified in Case 1. This group features a normal Sylow \( 3 \)-subgroup that is a direct product of two minimal normal subgroups: one of order \( 3 \) and the other of order \( 9 \). The Sylow \( 2 \)-subgroup of \( G \) acts Frobeniusly on the minimal normal subgroup of order \( 9 \), while its action on the minimal normal subgroup of order \( 3 \) is neither Frobenius nor trivial. Consequently, this case warrants careful consideration.
Based on Lemma \ref{thA} and Theorem B, it is reasonable to pose the following question:

{\bf Question:} Let $G$ be a finite group and $V $ represent the set of all isolated vertices of $\Gamma_{nv}(G)$. Can we conclude that there is at most one $\chi \in V$ for which $\Van(\chi) \neq \Van(G)$?

\smallskip

To demonstrate the solvability of the groups mentioned in Theorem B, we  prove Theorem C, which was established before Theorem B.
\begin{ThmC}
	Let $G$ be a finite non-abelian  group where $\Gamma_{nv}(G)$ is triangle-free. Then $G$ is solvable. 
\end{ThmC}

We were unable to identify any solvable groups where the common non-zero graph is triangle-free and contains more than one edge. Therefore, we believe that the following conjecture may hold true:

{\bf Conjecture 1. } Let $G$ be a finite non-abelian group with triangle-free common non-zero graph. Then $\Gamma_{nv}(G)$ has at most one edge.

 To provide evidence for Conjecture 1, we analyzed groups characterized by a common non-zero graph displaying a star shape, and we present the following finding:
\begin{ThmD}
Let $G$ be a finite group and $\Gamma_{nv}(G) $ be  an star. Then $G$ has exactly one non-linear irreducible character.  
\end{ThmD}
%For natural number $n$ and a set of primes $\pi$  we denote by $n_{\pi}$, the $\pi$-part of $n$ and $n_{\pi'}$  the $\pi'$-part of $n$.  Let $q$ be a prime, by $q$-singular element,   we mean an element whose order is divided by $q$.    Let $G$ be a group and $N$ be a normal subgroup of $G$.  If $g\in G$, we denote by $g^G$ a conjugacy class of $G$ containing $g$. 
%%%%%%%%%%%%%%%%%%
  If  $\chi\in \Irr(G)$ and  $k$ is an integer we denote by  $G^k$, the direct product of $k$ copies of $G$  and by   $\chi^k$  we mean  the product of $k$ copies of $\chi$, which is  $\chi\times \dots\times \chi\in \Irr(G^k)$.  We denote by $\Irr(G|N)$ the set of irreducible characters whose kernels does not contain $N$.  Also if $\theta\in \Irr(N)$ by  $\Irr(G|\theta)$ we mean the set of irreducible characters of $G$ above $\theta$. Let \( q \) be a prime number and \( n \) be an integer. We define \( n_q \) as the \( q \)-part of \( n \), which represents the largest power of \( q \) that divides \( n \). Similarly, \( n_{q'} \) denotes the \( q' \)-part of \( n \), calculated as \( n/n_q \). By a star graph of order $n$, sometimes simply known as an $n$-star, we mean a tree on $n$ vertices with one vertex having  degree $n-1$ and the other $n-1$ having  degree 1. By this definition  a  $1$-star is a graph with exactly one vertex.   
  All the other notations are standard and we  refer to \cite{Isaacs}, for instance.

\section{Preliminaries}

Before we delve into the proofs,  first address some preliminary lemmas that will be useful in our forthcoming discussions.

\begin{lemma}
	(see \cite[Corollary 2.2]{ono})\label{ono} Let $G$ be a non-abelian finite simple group and $p$ be a prime.  Then there exists $\chi\in \Irr(G)$ of $p$-defect zero, unless one of the following occurs:
	
	1) $p=2$ and $G\in \{  M_{12}, M_{22}, M_{24}, J_{2}, HS, Suz,  Ru, C_{1}, C_{3},BM, A_n \}$, where $n\not = 2m^2+m+2$ nor $ 2m^2+m$ for some integer $m$.
	
	2) $p=3$ and $G\in \{ Suz, C_3, A_n\}$, with $3n+1=m^2r$, where  $r$ is an square-free integer, divisible by some prime  $q\equiv 2\pmod 3$.
\end{lemma}

%\begin{lemma}(see \cite[Lemma 2.2]{brough})\label{brough}
%	Let $G$ be a finite group, $N$ a normal subgroup of $G$ and $p$ a prime. If $N$ has an
%	irreducible character of $p$-defect zero, then every $p$-singular element of $N$  is a vanishing element
%	in $G$.
%\end{lemma}

\begin{lemma}(see \cite[Lemma 5]{ema})\label{multi}
	Let $G$ be a finite group, and $N = S_1\times\dots\times S_k$ a minimal normal subgroup of
	$G$, where every $S_i$ is isomorphic to a non-abelian simple group $S$. If $\theta\in  \Irr(S)$ extends to $\Aut(S)$, then
	$\phi=\theta\times\dots\times\theta\in \Irr(N)$ extends to $G$.
\end{lemma}
\begin{lemma}(see \cite[Theorem 4.8]{ourself})\label{faithful}
	Let \( N \subseteq G \subseteq \Aut(N) \) be a minimal normal subgroup of $G$, where \( N \cong N_1 \times \dots \times N_r \cong S^r \)  and \( S \) is a finite non-abelian simple group. Then \( \Van(N) \subseteq \Van(G) \).  
	%\in \{M_{12}, M_{22}, M_{24}, Suz, Ru, HS, J_2, C_1, C_3, BM, A_n\} \) with \( n \geq 5 \).
\end{lemma}

The following results are taken from \cite[Theorems 1-4]{ema}. 
\begin{lemma}\label{extend}
	Let $S$ be a finite non-abelian simple group. Then 
	there is  a non-linear irreducible character which is extendible to	 
	$\Aut(S).$	  
\end{lemma}
%%%%%%%%%%%%%%%%%%%%%%%%5
\begin{lemma}\label{taze}
	Let \( N \subseteq G \subseteq \Aut(N) \) be a minimal normal subgroup of $G$, where \( N \cong N_1 \times \dots \times N_r \cong S^r \)  and \( S \) is a finite non-abelian simple group. Assume $I\subseteq \Irr(S)$ is a set of nontrivial  extendible characters of $S$ to $\Aut(S)$. Let $I$ induce a complete subgraph in $\Gamma_{nv}(S)$. Then the set $\{\chi_0\in \Irr(G)| [{\chi_0}_N,\chi^r]\not=0, \text{\ for some } \chi\in I\}$ induces a complete subgraph of $\Gamma_{nv}(G)$. 
\end{lemma}
\begin{proof}
	Let $\chi, \theta \in I$ be two distinct characters. Consequently, there exists $g_{\chi, \theta} \in \Van(S)$ such that $\chi(g_{\chi, \theta}) \theta(g_{\chi, \theta}) \neq 0$. Thus, we have $(g_{\chi, \theta}^r) = (g_{\chi, \theta}, \dots, g_{\chi, \theta}) \in \Van(N) \subseteq \Van(G)$, as stated in Lemma \ref{faithful}. Therefore, $\chi^r((g_{\chi, \theta}^r)) \theta^r((g_{\chi, \theta}^r)) = \chi(g_{\chi, \theta})^r \theta(g_{\chi, \theta})^r \neq 0$. 
	
	Now, let’s assume $\chi_1, \theta_1 \in \Irr(G)$ such that $[\chi_1, \chi^r] \neq 0$ and $[\theta_1, \theta^r] \neq 0$ for some distinct characters $\chi, \theta \in I$.  Then ${\chi_1}_N$ and ${\theta_1}_{N}$ are multiples of $\chi^r$ and $\theta^r$, respectively, as  $\chi^r$ and $\theta^r$ can be extended to $G$, as indicated by Lemma \ref{multi}.  Then, it follows that $\chi_1((g_{\chi, \theta}^r)) \theta_1((g_{\chi, \theta}^r)) \neq 0$. 
	
	If $\chi_1, \chi_2 \in \Irr(G|\chi^r)$ are distinct characters for some $\chi \in I$, then $\chi_1((g,\dots,g)) \chi_2((g,\dots,g)) \neq 0$ for all $g\in \Van(S)\cap \Nv(\chi)$. If $\Van(S)\cap \Nv(\chi)=\emptyset$, then $S$ is a $\mathcal{H}_1'$-group which is not possible by \cite[Theorem C]{ourself}. 
		Thus, the proof is complete.  
\end{proof}
%\begin{theorem}
%	Let $G$ be a finite group and $\Gamma_{nv}(G)$ is triangle-free,  then $G$ is solvable. 
%\end{theorem}

\begin{lemma}\label{newchar}
	Let $G$ be a finite group and $\theta\in \Irr(G|G')$ such that $\theta(g)\not =0$ for some $g\in G-G'$. Then there is a $\lambda\in \Irr(G/G')$ such that $ \theta\not = \lambda\theta$. Moreover if $gG'$ is an element of  order  $n>2$ in $G/G'$, then there is at least two characters $\lambda_1, \lambda_2\in \Irr(G/G')$ such that $\theta, \lambda_1\theta, \lambda_2\theta$ are three distinct irreducible characters. 
\end{lemma}
\begin{proof}
Clearly there exists $\lambda\in \Irr(G/G')$ such that $g\not\in \ker\lambda$. Therefore, $\lambda(g)\theta(g)\neq\theta(g)$. This confirms the first part of the lemma. Now, let's assume that the order of $gG'$ in $G/G'$ is $n>2$. We define $\lambda_1=\lambda$ and $\lambda_2=\overline{\lambda}$, where $\lambda\in \Irr(G/G')$ and $\lambda(g)=\epsilon$, with $\epsilon$ being the primitive $n$th root of unity. Thus, we have $\lambda_1\theta(g)=\epsilon \theta(g)$ and $\lambda_2\theta(g)=\overline{\epsilon}\theta(g)$, leading us to the desired result.
\end{proof}
 We use the following Remark repeatedly in the proofs without further reference:
\begin{remark}\label{rem}
	Assume a finite group $H$ acts on a finite group $K$. We say this action is Frobenius, if $C_H(k)=1$ for all $k\in K$. There is a well-known theorem which says: If $G=K\rtimes H$, then the following are equivalent:
	
	1. $C_K(h)=1$ for all $h\in H$;
	
	2. $C_H(k)=1$ for all $k\in K$;
	
	3. $G$ is a Frobenius group with Frobenius kernel $K$ and Frobenius complement $H$. 
\end{remark}

\begin{lemma}\label{Frobenius}
	 Let \( G \) be a finite group and \( N \) be a normal subgroup of \( G \) such that  \( G/N \) is cyclic. Suppose that for each \( g \in G - N \), it holds that \( |C_G(g)| = |G/N| \). Then, \( G \) is a Frobenius group with Frobenius kernel \( N \). Furthermore, if \( G/N \) has a prime order and there exists an element \( h \in G - N \) such that \( |C_G(h)| = |G/N| \), then   \( G \) is a  Frobenius group with Frobenius  kernel \( N \).	
\end{lemma}
\begin{proof}
We start by affirming that \( (|N|, |G/N|) = 1 \). Let's assume there is a prime \( q \) that divides both \( |N| \) and \( |G/N| \). Let \( Q \) be a Sylow \( q \)-subgroup of \( G \). The intersection \( Q \cap N \) is a normal subgroup of \( Q \). Therefore, we have \( N \cap \mathbf{Z}(Q) = Q \cap N \cap \mathbf{Z}(Q) \neq 1 \). 

Now, consider an element \( g \) in \( Q - N \). We have  \( |C_G(g)| = |G/N| \). However, since \( N \cap \mathbf{Z}(Q) \) is not trivial, we have that \( 1 \neq N \cap \mathbf{Z}(Q) \subset C_G(g) \). This leads to a contradiction, because \( |QN/N| \cdot |N \cap \mathbf{Z}(Q)| \) must divide \( |C_G(g)|_q = |G/N|_q=|QN/N| \), given that \( QN/N \) is cyclic. 

From this, we conclude that \( G = N \rtimes H \), where \( H \cong G/N \) is a cyclic group. Consequently, we find that \( C_G(g) \cap N = 1 \) for all \( g \in H \). Thus, \( G \) is a Frobenius group, with \( N \) as the Frobenius kernel.

In the second part of the lemma, if \( |G/N| \) is a prime \( p \) and there exists an element \( h \in G - N \) such that \( |C_G(h)| = p \), then it follows that  $h$ has order $p$ and \( G/N \cong \langle h \rangle \) is a group of order \( p \). Moreover $G=N\rtimes \langle h \rangle$ and so using Remark \ref{rem},  $G$ is a Frobenius group with Frobenius kernel $N$.
\end{proof}

\begin{lemma}\label{orbit}	Let $G$ be a finite group acting on a module $M$ over a finite field and $1 \in S \subseteq M $ be a union of some $G$-orbits of $M$.  Assume  $y\in G$ is an element of order $2$ and   for every $v \in M -S $,  $C_G(v)=\langle y^g\rangle$, for  some $g\in G$. Assume for each $1\not =x\in S$, $(|C_G(x)|, 2)=1$.  Then, $|M| -|S| =|y^G| (|C_M ( y )|- 1)$
\end{lemma}
\begin{proof}
 Set $M^{\#}=M-1$. By the hypothesis $C_{M^{\#}}(y^g)\subseteq M-S$, for every $g\in G$ and  $ M-S \subseteq  \cup_{g\in G} C_{M^{\#}(y^g)}$. Thus $M-S=\cup_{g\in G} C_{M^{\#}}(y^g)$. On the other hand $\Omega=\{C_{M^{\#}}(y^g)| g\in G\}$  forms a partition for $M-S$. As $|\Omega|=|y^G|$ we get that $|M-S|=|y^G||C_{M}(y)-1|$ as desired. 
	%%%%%%%%%
	%x∈I (A5) CN♯ (< x >), where I (A5) is the set of
	%involutions of A5. It is well-known that all elements in I (A5) are A5-conjugate. So if
	%CN♯ (< x >) ̸= 1 for some x ∈ I (A5), then CN♯ (< xg >) ̸= 1, where g is an element
	%of A5. Therefore ∪
	%x∈I (G) CN♯ (< x >) = N♯
	%. Also CN♯ (< x >)’s forms a partition for
	%N♯
	%, for all x ∈ I (A5). Hence, (p
	%α − 1)/(p
	%α0 − 1) = |A5 : CA5
	%(< x >)| = 15,
\end{proof}
\begin{lemma}\label{pgroup}
	Let $G$ be a $p$-group  such that $|G|/|{\bf Z}(G)|=p^{2n+1}$ for some integer $n$. Then there is no $\chi\in \Irr(G|{\bf Z}(G))$ with  $\Van(\chi)=\Van(G)$. 
\end{lemma}
\begin{proof}
	Assume  there is  $\chi\in \Irr(G|{\bf Z}(G))$ with  $\Van(\chi)=\Van(G)$. Then $G-{\bf Z}(G)=\Van(\chi)$, by \cite[Lemma 2.9]{BDS}.  Thus by \cite[Problem 6.3]{Isaacs} we get that $\chi(1)=\sqrt{p^{2n+1}}$, which is impossible. 
\end{proof}
\begin{lemma}\label{extraspecial}
Let \( G \) be a finite \( p \)-nilpotent group, where \( p \) is a prime. Assume \( H \) is the Hall \( p' \)-subgroup of \( G \), and that \( G/H \cong P \in \Syl_p(G) \) is an extraspecial \( p \)-group. Also, assume \( G/G' \cong P/P' \) and  $G-G'\subseteq \Van(\chi)$ for all $\chi\in \Irr(G|G')$. If   there exists an element \( x \in G' - H \) such that \( \chi(x) = 0 \) for all \( \chi \in \Irr(G|H) \) then we find that \( P \cong Q_8 \) and \( G \) forms a Frobenius group with \( H \) as its Frobenius kernel. In particular, \( H \) is abelian.
	\end{lemma}
\begin{proof}
The assumption indicates that all nonlinear irreducible characters, except for the ones in \( \Irr(G/H)=\Irr(P) \) are vanishing on $x$. By applying the second orthogonality relation (see \cite[Theorem 2.18]{Isaacs}), we can conclude that \( |C_G(x)| = \sum_{\theta\in \Irr(G/H)} |\theta(x)|^2 =|P|\).  The last equality holds  because $xH\in G'/H\cong P'={\bf Z}(P)$.  This shows that $x$ is a $p$-element. So we may assume $x\in {\bf Z}(P)=P'\subseteq G'-H$.   Now, consider \( y \in P - {\bf Z}(P) \). Applying the second orthogonality relation on $y
$ we get that $|C_G(y)|=\sum_{\theta\in \Irr(G/G')} |\theta(y)|^2 =|P|/|P'|$. 
Therefore, we have \( C_G(y) \subseteq C_G(x) \subseteq P \), which shows that \( P \) is acting Frobeniusly on \( H \). Since $P$ is an extraspecial $p$-group (not a cyclic group), we conclude that $G$ is a generalized quaternion group and the only possibility is  \( P \cong Q_8 \). Hence as $G$ has a Frobenius complement of even order, then  the Frobenius kernel of $G$, means $H$,  is abelain. 	
\end{proof}

%In the following remark, we highlight some notable values of non-linear irreducible characters when the group \( G \) possesses exactly one such character. For a detailed characterization of groups with precisely one non-linear irreducible character, refer to \cite{seitz}. We use this remark without further reference. 

%\begin{remark}\label{remark}
%	(a) Let \( G \) be an extraspecial \( 2 \)-group and let \( \chi \) be the sole non-linear irreducible character of \( G \). Then, for every \( g \in {\bf Z}(G) \) where \( 1 \neq g \), it holds that \( \chi(g) = -\sqrt{|G:G'|} \). 
	
%	(b) Let \( G \) be a Frobenius group of order \( p^a(p^a-1) \) for some integer \( p \) and prime \( a \), where the Frobenius kernel is an elementary abelian group of order \( p^a \) and the Frobenius complements are cyclic. Let \( \chi \) be the only non-linear irreducible character of \( G \). Then, for all non-trivial element  \( g \) in the Frobenius kernel of \( G \), it follows that \( \chi(g) = -1 \).   
%	\end{remark}
\begin{lemma}\label{pp}
	Let $G$ be a Frobenius group. %with a cyclic Frobenius group 
	% of order $p^a-1$ for some prime $p$ and integer $a$.
	 Assume the Frobenius kernel of $G$ is $P\in \Syl_p(G)$ and $G$ has a minimal normal subgroup $K\subset P$, where $P/K$ is a chief factor of $G$ with order $p^a$ and $|K|=p^b$, for some integers $a$ and $b$. Let $\Van(G)\cap P\not=\emptyset$ and for all $\chi\in \Irr(G|K)$ we have $\Van(G)\cap P\subseteq \Van(\chi)$. Then $\Van(G)\cap P\subseteq (P-K)$ and $b<a$.  
\end{lemma}
\begin{proof}
Let  $g
\in \Van(G)\cap P$. As $K$ is a minimal normal subgroup of $G$ we find that $K\subseteq {\bf Z}(P)$. If   \( g \in  K \cap \Van(G) \), then  by the second orthogonality relation we have: 

$$p^{b+a}=|C_{G}(g)|=\sum_{\theta\in \Irr(G/K)}|\theta(g)|^2=|C_{G/K}(gK)|=|P/K|=p^a,$$ which is a contradiction. Hence, we can assume \( g\in \Van(G)\cap P\subseteq   P - K \). By applying the second orthogonality relation to \( g \)  we derive that:
\[
p^b = |K| < |C_{G}(g)| = \sum_{\theta \in \Irr(G/K)} |\theta(g)|^2 = |C_{G/K}(gK)|=|P/K|= p^a,
\]
where the first inequality holds because \( K \subseteq \mathbf{Z}(P) \). Consequently, we find that \( b < a \). 	
\end{proof}
\begin{lemma}\label{pq}
	Assume $G$ has a normal and abelian  subgroup  $N={\bf Z}(G)\times K$, where $K$ is a Sylow $p$-subgroup of $G$ for some prime $p$.  Then $N\subseteq \Nv(G)$
\end{lemma}
\begin{proof}
	Clearly ${\bf Z}(G)\subseteq \Nv(G)$. Moreover, $K$  is a normal and abelian  Sylow subgroup of $G$ and so  irreducible  character degrees of $G$ are coprime to $|K|$, implying that $K\subseteq \Nv(G)$, by applying \cite[Corollary 2.2]{BDS}. Assume $xy\in \Van(\chi)$, for some $\chi\in \Irr(G|G')$, where $x\in {\bf Z}(G)$ and $y\in K$. Assume $\chi\in \Irr(G|\gamma\times \alpha)$, for some $\gamma\in \Irr({\bf Z}(G))$ and $\alpha\in \Irr(K)$.  Then we have:
	$$\chi_N=b\sum_{i=1}^t (\gamma^{g_i}\times \alpha^{g_i} )=\gamma\times (b\sum_{i=1}^t \alpha^{g_i}),$$ where
	$\{g_1, \dots, g_t\}$ is the set of right  transversal of $I_G(\gamma\times \alpha)=I_G(\gamma)\cap I_G(\alpha)=I_G(\alpha)$ in $G$ and $b$ is an integer. Therefore we conclude that $\chi(xy)=\gamma(x)(b\sum_{i=1}^t \alpha(y)^{g_i})=0$, this implies that $\chi(y)=\gamma(1)\times (b\sum_{i=1}^t \alpha(y)^{g_i})=0$ contradicting that  $K\subseteq \Nv(G)$.   
\end{proof}
\section{Main Results}

To begin this section, we will focus on proving Theorem C, as its proof stands independently from the other key theorems. It's important to note the following remark that will aid in understanding the proof of Theorem C.

\begin{remark}\label{an}
	It is well-known that irreducible characters of $S_n$ are parameterized by
	partitions of $n$, and so we use  
	$\chi^{\lambda}$
	to denote the character corresponding to a partition
	$\lambda$. Moreover, looking at \cite{JK}, we get that for $n\geq 9$, 	$(\chi^{(n-i, 1^i)})_{A_n}\in \Irr(A_n)$ for $i\in \{1,2,3\}$.  Furthermore,  according to Murnaghan-Nakayama formula (see \cite[2.4.7]{JK}) we see that all of these three characters are not vanishing on cycles of length $5$, while cycles of length five are in $\Van(A_n)$. 
\end{remark}

{\bf The proof of Theorem C.}
Assume that $G$ is a  non-solvable group where $\Gamma_{nv}(G)$ is triangle-free.  Let $S$ denote the solvable radical  subgroup of $G$. Consider $R/S$, which is a chief factor of $G$, and define $C/S = C_{G/S}(R/S)$. This leads to the inclusion $R/S\cong RC/C \subseteq G/C \subseteq \Aut(R/S)$. %Consequently, the graph $\Gamma_{nv}(G/C)$ will have at least two vertices that are not adjacent, as it is triangle-free. This indicates that either $G/C$ is a strongly $\mathcal{H}_2'$-group, or there exists a normal subgroup $M/C$ such that $\Gamma_{nv}(G/M)$ has at most two non-linear irreducible characters. In the case where $|\Irr(G/M|G'M/M)| = \{\chi, \theta\}$, it must hold that the characters $\chi$ and $\theta$ are adjacent in $\Gamma_{nv}(G/C)$.
Let \( R/S\cong R_1 \times \dots \times R_k \), where each \( R_i \) is isomorphic to the non-abelian simple group \( R_0 \). 
Initially,  consider the case where $R_0$  is one of the groups listed in Lemma \ref{ono}. Then looking at Atlas of finite groups \cite{atlas}  and Remark \ref{an}  we can find a set $I\subseteq \Irr(R_0)-\{1_{R_0}\}$  such that $|I|\geq 3$ and $I$ satisfies the hypothesis of Lemma \ref{taze} and so $\Gamma_{nv}(G)$ has  a complete subgraph with at least three vertices.

Next, let's examine the case where $R_0$ is a finite simple groups not listed in Lemma \ref{ono}. 
By utilizing Lemmas \ref{ono} and \ref{faithful}, we see that the set $RC/C - C/C$ is contained within $\Van(G/C)$, which in turn is contained in $\Van(G)$. Assume $\gamma\in \Irr(R_0)$ is an extendible character to $\Aut(R_0)$, whose existence is guaranteed by Lemma \ref{extend}. Notably, $\gamma^k\in \Irr(RC/C)$ extends to a character $\chi \in \Irr(G/C)$, by Lemma \ref{multi}. First assume $|G/RC|>1$. Thus there exists at least two characters above $\gamma^k$, say  $\chi$ and $\chi_0$ and according to Lemma \ref{taze} they are adjacent. Moreover $\Van(\chi)\cap RC/C=\Van(\chi_0)\cap RC/C$, as $(\chi_0)_{RC/C}$ is a multiple of $\gamma^k$.   Now, take $1 \neq \eta = \eta_0^k \in \Irr(RC/C)$ to be a character such that $ \gamma \not =\eta_0\in \Irr(R_0)$. Let $\theta \in \Irr(G/C|\eta)$. Then  $(\Nv(\chi)\cap RC/C)-\{C/C\} \subseteq (\Van(\theta) \cap RC/C)$, or else the set $\{\chi, \chi_0, \theta\}$ would form a triangle. This leads to the conclusion that $\frac{\theta(1)\chi(1)}{|RC/C|} = [\theta_{RC/C}, \chi_{RC/C}] = [\theta_{RC/C}, \gamma^k] = 0 $, which ultimately presents a contradiction. 

Thus we may assume $|G/RC|=1$ and so $R/S\cong RC/C=G/C\cong R_0$. But in those cases $\Gamma_{nv}(G/C)$ is complete according to \cite[Corollary E]{ourself}, which is a final contradiction.   $\blacksquare$
%\end{proof}	

% implying that they are adjacent to $\theta$, unless $S_0\cong \PSL_2(3^f)$ for some integer $f$. First assume $k>1$. Then $\eta_0=St\times \dots\times St$ and $\eta_1=St\times\dots St \times 1$ both are invariant under $\Aut(S)^k$. So it is easy to see that $\Nv(\chi_0)\cap \Nv(\chi_1)$  is not empty, where $\chi_0\in \Irr(G|\eta_0)$ and  $\chi_0\in \Irr(G|\eta_0)$. Thus $\theta$ is adjacent to $\chi_0$ and $\chi_1$, which is impossible.  Thus $k=1$ and so $G/C\subseteq \Aut(\PSL_2(^f))$. Thus $G/R$ is abelian, while $G/C$ must have a normal subgroup $M/C$ such that $G/M$ is nonabelian solvable group, a contradiction. 

\bigskip

In the squeal, we establish some lemmas to facilitate a clearer understanding of proof of Theorem A.

\begin{lemma}\label{Frob}
	Let $G$ be a Frobenius group whose  Frobenius kernel is a $p$-group and the  Frobenius complement is a cyclic group. Assume that $\Gamma_{nv}(G)$ is null, then $G$ is a  strongly $\mathcal{H}_1'$-group. 
\end{lemma}
\begin{proof}
	Assume that \( G \) is not classified as a strongly \( \mathcal{H}_1' \)-group.  Lemma \ref{thA}  indicates that there exists a unique character \( \theta \in \Irr(G|G') \) such that \( \Van(\theta) \subset  \Van(G) \). According to the hypothesis $G'$ is the Frobenius kernel of $G$ and it is a $p$-group. Let \( K \) be a normal subgroup of \( G \) contained within \( G' \), and $G'/K$  is a chief factor of $G$. Define \( \overline{G'} = G'/K \). 
	
	It is evident that the set \( G - G' \) is included within \( \Van(\chi) \) for all non-linear characters \( \chi \in \Irr(G) \), since \( G \) is a Frobenius group with \( G' \) as its  Frobenius kernel.
	
	Now, for any non-linear irreducible character \( \eta \) of \( \overline{G} = G/K \), we note that \( \overline{G} - \overline{G'} = \Van(\eta) \). Given that no character degree of \( \overline{G} \) is divisible by \( p \), it follows that \( \Van(\overline{G}) \cap \overline{G'} = \emptyset \), by  \cite[Corollary 2.2]{BDS}. It is important to remember that for any character \( \chi \in \Irr(G|G')-\{\theta\} \), we also have \( \Van(G) = \Van(\chi) \). Under this assumption, \( \Irr(\overline{G}|\overline{G'})=\{\theta\} \).  This leads us to conclude that \( |\overline{G}/\overline{G'}| = |\overline{G'}| - 1 = p^a - 1 \) for some integer \( a \), as referenced in \cite{seitz}. %Consequently, we establish that \( \theta(g) = -1 \) for every element \( g \) in \( \overline{G'} \). Indeed, applying the second orthogonality relation for all \( g \in \overline{G'} \), we find the equation 
	%%%%%%%%%%%%%%%%%%%%%%%%%
	%\[ \sum_{\chi\in \Irr(\overline{G}/\overline{G'})} \chi(1)\chi(g)+ \theta(1)\theta(g)=
	%|G:G'| + |G:G'|\theta(g) = 0,
	%\]
	%which leads to \( \theta(g) = -1 %\). 
	If \( K \) is trivial, then \( G \) would be a strongly \( \mathcal{H}_1' \)-group, concluding our argument. Thus, we will proceed under the assumption that \( K \) is non-trivial.
	
	Next, let \( M \) be a normal subgroup of \( G \) contained in \( K \), such that \( K/M \) is a chief factor of \( G \). We will  focus attention  to \( \tilde{G} = G/M \). Clearly, \( \tilde{K} \subseteq \mathbf{Z}(\tilde{G'}) \), implying that for all \(1\not= x\in \tilde{K} \), we have \( C_{\tilde{G}}(x) = \tilde{G'} \). 
	Assuming \( |\tilde{K}| = p^b \) for some integer \( b \), since for all non-trivial elements \( x \in \tilde{K} \) we have \( |x^{\tilde{G}}| = |G:G'| \), it follows that \( p^b - 1 = k(|G:G'|) = k(p^a - 1) \), for some integer $k$, indicating that \( b \) must be a multiple of \( a \).
	
	Note that  \( \tilde{G'} \cap \Van(\tilde{G})\) is not empty, as otherwise the vanishing set of all characters in $\Irr(\tilde G|\tilde{G'})$ matches vanishing set of $\theta$, which is a contradiction. Thus   all characters in \( \Irr(\tilde{G}|\tilde{K}) \) vanish on \( \tilde{G'} \cap \Van(\tilde{G})\). Hence we can apply Lemma \ref{pp} and consequently, we find that \( b < a \), leading to a contradiction, as $b$ is a multiple of $a$.
\end{proof}
\begin{lemma}\label{yash}
Let $G$ be a finite group. If $\Gamma_{nv}(G)$ is null, then
\begin{itemize}
\item[(i)] $G$ is a strongly $\mathcal{H}_1'$-group; or 
\item[(ii)] Every character degree in  $\Irr(G|G')$ is divided by a prime  $p$. In particular $ G $ is $ p $-nilpotent;
\end{itemize}
\end{lemma}
\begin{proof}
Suppose that $G$ is not a strongly $\mathcal{H}_1'$-group. Then by Lemma \ref{thA}, there exists only one  character $ \theta \in \Irr(G|G') $ such that $ \Van (\theta)\not= \Van(G) $. Let $ g\in \Van(G) $ be a $ p $-element for some prime $ p $. Then $ \chi(g)=0 $ for all non-linear $ \chi \in \Irr(G)-\{ \theta\} $. This means that $ p\mid \chi(1) $ for all non-linear $ \chi \in \Irr(G) -\{ \theta\} $, by \cite[Corollary 2.2]{BDS}. If $ \theta(g)=0 $, then using \cite[Corollary 12.2]{Isaacs}, $ G $ is $ p $-nilpotent and so (ii) holds. Assume that $ \theta(g)\not= 0 $ and that $ p\nmid \theta(1) $. Then $ G $ has exactly one non-linear irreducible character $ \theta $ of $ p' $-degree. Using \cite[Theorem A]{KB}, $\ker\theta\subseteq G'$ and $ G/\ker \theta $ is isomorphic to a Frobenius group with a Frobenius kernel of order $ p^{n} $ and a cyclic Frobenius complement of order $ p^{n}-1 $ for some integer $n$. Note that every non-linear irreducible character vanishes on all the non-trivial elements in $G-G'$. Hence, by the second orthogonality relation we have $|C_G(g)|=\sum_{\chi\in \Irr(G/G')}|\chi(g)|^2=|G:G'|$, for every $g\in G-G'$. Thus applying Lemma \ref{Frobenius},  $ G $ is a Frobenius group with Frobenius kernel $G'$. This means that $ G' $ is nilpotent. But by \cite[Proposition 2.2]{KB}, $ N_{G'}(P)=P $, where $ P $ is a Sylow $ p $-subgroup of $ G $. Hence $ P=G' $ and so  by Lemma \ref{Frob} we get that $G$ is a strongly $\mathcal{H}'_1$-group, a contradiction.
\end{proof}

{\bf The proof of Theorem A.}
	Let $G$ be a counterexample with minimal order. Then by Lemma \ref{yash},  for every $\chi\in \Irr(G|G')$, $p$ divides $ \chi(1)$ and so $G$ is a $p$-nilpotent group, for some prime $p$. Since $G$ is not a strongly $\mathcal{H}_1'$-group, there exists a character $\theta\in \Irr(G|G')$ such that $\Van(\theta) \subset \Van(G)$ and for every other non-linear irreducible character $\chi$, we  have that $\Van(G) = \Van(\chi)$, by Lemma \ref{thA}. We will prove the theorem by the  following  steps:
	\smallskip
	
	{\bf Step 1.} If $h\in \Van(G)-G'$, then $h\in \Van(\chi)$ for all $\chi\in \Irr(G|G')$. 
	
	\smallskip
	
	Assume otherwise. Then  $\theta(h)\not =0$. Thus for every $\zeta\in \Irr(G/G')$ we conclude that $\Van(\zeta\theta)=\Van(\theta)$, which is a contradiction, as $\theta\not =\zeta\theta$ for some $\zeta\not =1$, by Lemma \ref{newchar}. 
	
	\smallskip
	
	{\bf Step 2.}  $G$ is not a $p$-group. 
	
	 \smallskip
	
	On the contrary assume $G$ is a $p$-group and $K$ is an arbitrary  minimal normal subgroup of $G$.  Then either $G/K$ is abelian or as $\Gamma_{nv}(G/K)$ is also null and $|G/K|<|G|$ we have that $G/K$ is a strongly $\mathcal{H}_1'$-group.

	 Assume the second case occurs. Firstly, let   $ \theta\not \in \Irr(G/K)$. Then all characters of $G/K$ are vanishing on $\Van(G)=\Van(G/K)=G/K-{\bf Z}(G/K)$. Moreover, according to Lemma \ref{ThmB}, $G/K/{\bf Z}(G/K)$ is an elementary abelian group. So $KG'/K\subseteq {\bf Z}(G/K)$.  
	 Hence there exists $h\in (G-G')\cap \Van(G)$ such that $\theta(h)\not =0$, a contradiction by Step 1. Thus $G/K$ has exactly one non-linear irreducible character $\theta$ (as it is a strongly $\mathcal{H}_1'$-group).   So by the main result of  \cite{seitz}, we get that $G/K$ is an extraspecial $2$-group and so $|G/K|=2^{2n+1}$ for some integer $n$. Thus $\Van(\theta)=G/K-Z/K$, where	
	 $Z/K={\bf Z}(G/K)$. As $K\subseteq {\bf Z}(G)$ and $\Van(G)=G-{\bf Z}(G)$,  there exists $h\in Z-K$ such that all  non-linear characters, except for $\theta$, are vanishing on it. As $|Z/K|=2$, we get that $K={\bf Z}(G)$ and so $\Van(G)=G-K$. Hence we get that all  non-linear characters in  $\Irr(G|K)$, are vanishing on $G-K$ and by Lemma \ref{pgroup} we get a contradiction.

	 Assuming that \( G/K \) is abelian, we can conclude that \( G' = K \). Furthermore,  \( K = G' \subseteq \mathbf{Z}(G) \),  is the unique minimal normal subgroup of \( G \). According to \cite[Theorem 12.3]{Isaacs}, we find that \( \text{Van}(G) = \text{Van}(\chi) = G - G' \) for all \( \chi \in \text{Irr}(G | G') \). This implies that $G$ is  a strongly \( \mathcal{H}_1' \)-group, a contradiction.

\smallskip

 {\bf Step 3.} The Sylow $p$-subgroups of $G$ are not abelian. 
 
 \smallskip
 
Assume, on the contrary, that \( P \in \Syl_p(G) \) is abelian. Let \( K \) represent the group outlined in part (2) of Lemma \ref{thA}. We note that \( \Irr(G/K|G'K/K) = \{\theta\} \), \( \Van(G) \cap G'K \neq \emptyset \), and for all non-linear characters \( \chi \in \Irr(G|K) \), it holds that \( \Van(G) = \Van(\chi) \).  
Consequently, this leads us to conclude that either \( G/K \) is a Frobenius group or it is an extraspecial \( 2 \)-group, by \cite{seitz}. If the latter is the case, then \( G \) possesses an irreducible  character of degree \( 2^n \) for some integer \( n \), which indicates that \( p = 2 \), as all non-linear character degrees of $G$ are divided by $p$. Therefore, the Sylow \( 2 \)-subgroup cannot be abelian, as required.

Thus, we proceed with the assumption that \( G/K \) is a Frobenius group of order \( r^a(r^a-1) \) with a cyclic Frobenius  complement, for some prime \( r \) and integer \( a \). Note that $\Van(\theta)=G-G'K$.  Observe that \( G'K = G' \). If not, we could find a non-trivial linear character \( \lambda \in \Irr(G'K/G') \), with \( \lambda\theta\not =\theta \), by Lemma \ref{newchar},  as a non-linear irreducible  character of \( G \) with \( \Van(\theta) = \Van(\theta\lambda) \), a contradiction. Hence, \( G'K = G' \) and so $\Van(\theta)=G-G'K=G-G'\subseteq \Van(G)$.

By employing the second orthogonality relation, we obtain for all \( x \in G - G' \) that \( |G/G'| = \sum_{\chi \in \Irr(G/G')} |\chi(x)|^2 = |C_G(x)| \). Thus applying Lemme \ref{Frobenius}, we get that $G$ is a Frobenius group with Frobenius kernel $G'$. Then $G'$ is nilpotent. Given that $G$ is $p$-nilpotent, we find that,  $G/G' \cong P$, where $P \in \Syl_p(G)$, indicating that $P$ is cyclic.   Referring to Theorem \ref{Frob}, if we can establish that $G'$ is a $q$-group for some prime $q$, we will arrive at a contradiction. 

%Next, we establish that \( (|G'|, |G/G'|) = 1 \). Suppose a prime \( q \) divides both \( |G'| \) and \( |G/G'| \). Let \( Q \) be a Sylow \( q \)-subgroup of \( G \). The intersection \( Q \cap G' \) is a normal subgroup of \( Q \), and thus \( Q \cap G' \cap \mathbf{Z}(Q) \neq 1 \). Taking some \( h \in Q - G' \), we find \( 1 \neq Q \cap G' \cap \mathbf{Z}(Q) \subset C_G(h) \). This results is a contradiction since \( |QG'/G'| \cdot |Q \cap G' \cap \mathbf{Z}(Q)| \) must divide \( |C_G(h)| = |G/G'| \), given that \( QG'/G' \) is cyclic.
 
% this implies that $|C_G(h^{O(h)/q})|=|G/L|$ and so  $(|G/L|,|L|)=1$

% Therefore, we can deduce that $G/G'$ acts Frobeniusly on $G'$, which leads us to conclude that $G'$ is nilpotent.

 Assume the contrary. Then $G'=H_0 \times H_1$, where $H_0$ and $H_1$ are distinct Hall subgroups of $G'$. We can take normal subgroups $T_0 \subseteq H_0$ and $T_1 \subseteq H_1$ within $G$ such that $G'/T_iH_j = H_iH_j/T_iH_j$ serves as chief factors of $G$, for the pairs $(i,j) \in \{(0,1), (1,0)\}$. Consequently, for  non-linear character $\chi_i \in \Irr(G/T_iH_j)$, with $(i,j) \in \{(0,1), (1,0)\}$, we find that $\Van(\chi_i) = G - G' = \Van(\theta)$ holds true. This establishes a contradiction. Therefore, $G'$ cannot have more than one Hall subgroup, which confirms that $G'$ must indeed be a $q$-group for some prime $q$, as required. Thus, by Lemma \ref{Frob}, we arrive at a contradiction.

 % Note that all  non-linear characters in  $\Irr(G/O_{t'}(G'))-\{\theta\}$  are vanishing on  $G-G'\subset \Van(G)$, where $t$ is an arbitrary prime divisor of $|G'|$. If $t$ and $q$ are two distinct prime divisors of $|G'|$ and $\chi_t \in \Irr(G/O_{t'}(G'))-\{\theta\} $   and $\chi_q \in \Irr(G/O_{q'}(G'))-\{\theta\} $  
 %are non-linear, then $\Van(G)=\Van(\chi_q)=\Van(\chi_t)$. Note that $\Van(\theta)=G-G'\subset \Van(G)$.  Hence there exists $x\in G'$ such that $\chi_t(x)=\chi_q(x)=0$. Thus we may assume there exists $q$-element $x_q$,  $t$-element $x_t$ and $\{q,t\}'$-element $x_{\{q,t\}'}$ such that  $x=x_qx_tx_{\{q,t\}'}$.  As $\chi_t(x)=0$ we get that $\chi_t(x_t)=0$, implying that $\chi_q(x_t)=0$, a contradiction. Thus  $|G'|$ is a $q$-group for some prime $q$. Now using theorem \ref{Frob}, we get the final contradiction of this step. So without loss of generality we may assume $\Irr(G/O_{q'}(G')| G'/O_{q'}(G'))=\{\theta\}$ and so $O_{q'}(G')=K$. Then $G'=K\times O_q(G)$. Assume $T\subseteq O_{q'}(G')$ be a normal subgroup of $G$ such that  $O_{q'}(G')/T$ is a chief factor of $G$. Then  $\Irr(G/O_{q}T)$ contains irreducible characters .  
  \smallskip
 
 {\bf Step 4.} $G$ is a Frobenius group  whose complement is isomorphic to $Q_8$ and the kernel is abelian.
 
 \smallskip 
 
Let \( P \) denote the Sylow \( p \)-subgroup of \( G \). From Step 3 and the fact that $|P|<|G|$, we know that \( P \cong G/O_{p'}(G) \) is categorized as a strongly \( \mathcal{H}_1' \)-group. According to Lemma \ref{ThmB}, we have:
\[
\Van(G/O_{p'}(G)) = G/O_{p'}(G) - \mathbf{Z}(G/O_{p'}(G)) \subseteq G/O_{P'}(G) - G'O_{p'}(G)/O_{p'}(G).
\]
Assuming that \( \theta \not\in \Irr(G/O_{p'}(G)) \), there must exist an element \( h \in G - G'O_{p'}(G) \subseteq G - G' \) such that \( \theta(h) \neq 0 \), which contradicts our earlier findings in Step 1. Therefore, we can assert that \( P \cong G/O_{p'}(G) \) has precisely one non-linear irreducible character \( \theta \).

This leads us to conclude that \( P \) is an extraspecial \( 2 \)-group, as indicated in \cite{seitz}. Moreover, we find that
\[
\Van(\theta)  = G/O_{p'}(G) - \mathbf{Z}(G/O_{p'}(G)).
\]
Let \( Z/O_{p'}(G) = \mathbf{Z}(G/O_{p'}(G)) = P'O_{p'}(G)/O_{p'}(G) \cong P' \). It is clear that $P'$ is a  cyclic group of order $2$. We aim to establish that \( Z = G' \). It is evident that \( G' \subseteq Z \). If $Z\not =G'$, then  for some  \( \lambda \in \Irr(Z/G') \), we have $\lambda\theta\not =\theta$, by Lemma \ref{newchar}, while  \( \Van(\lambda\theta) = \Van(\theta) \),  a contradiction. Thus  we conclude \( Z = G' \).  It follows that \( G/G' \cong P/P' \) with \( G - G' \subseteq \Van(G) \).

Next, we assert that \( G' - O_{p'}(G) \subseteq \Van(G) \). Suppose, for the sake of contradiction, that this is not true. Let \( K \) be a normal subgroup of \( G \) such that \( O_{p'}(G)/K \) serves as a chief factor of \( G \). Clearly, the orders satisfy \( (|G'/O_{p'}(G)|, |O_{p'}(G)/K|) = 1 \).

If we have that \( C_{G'/K}(O_{p'}(G)/K) = O_{p'}(G)/K \), then, by Lemma 2.9 in  \cite{silvio1}, it follows that \( G' - O_{p'}(G) \subseteq \Van(G) \), leading to a contradiction. Therefore, we must have:
\[
C_{G'/K}(O_{p'}(G)/K) = G'/K.
\]
Consequently,
 \( G'/K \cong P_2/K \times O_{p'}(G)/K \), where \( P_2/K \cong G'/O_{p'}(G) \) is a
  cyclic group of order \( 2 \). Since \( \theta \notin \Irr(G/P_2) \) ($P_2\not \subseteq O_{p'}(G)=\ker\theta$),	
	 we can conclude that all non-linear irreducible characters of \( G/P_2 \) vanish on \( \Van(G) \). It is also important to note that \( G/P_2 \) is not abelian, as $P_2\subset G'$.
Now, taking \( \overline{G} = G/P_2 \), we have \( \Van(G) = \Van(\overline{G}) = \Van(\chi) \) for all non-linear irreducible characters $\chi$ of \( \overline{G} \). Given that \( G - G' \subseteq \Van(G) \), we find that \(\overline{G}-  \overline{G'} \subseteq \Van(\overline{G}) \). Consequently, for any \( x \in \overline{G} - \overline{G'} \), it holds that:
\[
|C_{\overline{G}}(x)|=\sum_{\chi \in \Irr(\overline{G}/\overline{G'})} |\chi(x)|^2=|\overline{G} : \overline{G'}|.
\]
This indicates that \( \overline{G} \) is a Frobenius group, and its Frobenius complement is isomorphic to \( G/G' \cong \overline{G}/\overline{G'} \cong P/\mathbf{Z}(P) = P/P' \). However, this situation is untenable since \( \overline{G} \) is a \( 2 \)-nilpotent Frobenius group with a Frobenius complement that is an elementary abelian \( 2 \)-group of order at least \( 4 \) (As $P$ is an extraspecial $2$-group). 
Thus, we have validated our claim and \( G' - O_{p'}(G) \subseteq \Van(G) \), implying that \( G - O_{p'}(G) \subseteq \Van(G) \). Now applying Lemma \ref{extraspecial}, we get that $G$ is a Frobenius group whose Frobenius complements are $Q_8$ and the Frobenius kernel is abelian. $\blacksquare$

 	\smallskip
 	
    In the following  we initially  prove some  lemmas and theorems that lead us to prove Theorems B and D. 
   
   \begin{lemma}\label{quaterniun}
   	Let $G$ be a finite group such that $\Gamma_{nv}(G)$ has exactly one edge $\{\theta_1, \theta_2\}$. Then there is no normal subgroup   $N$ of  $G$ such that $G/N\cong C_3^2:Q_8$ and $\{\theta_1, \theta_2\}\subseteq \Irr(G/N)$.   
   \end{lemma}
\begin{proof}
On the contrary assume $N$ is a normal subgroup of $G$ such that $G/N\cong C_3^2:Q_8$ and $\{\theta_1,\theta_2\} \subseteq \Irr(G/N)$. We may assume that $\theta_1(1)=2$ and $\theta_2(1)=8$. In this case $\Van(\theta_1)=G-G'N$ and $\Van(\theta_2)=G-K$, where $K/N$ is the Sylow $3$-subgroup of $G/N$.    Moreover all characters  $\eta\in \Irr(G|G')-\{\theta_1, \theta_2\}$ are vanishing on $\Van(G)\cap G'N$, as otherwise we have another edge started from $\theta_1$ to some irreducible characters besides $\theta_2$.

To continue, we need to establish that \( G'N = G' \). If not, for each \( \lambda \in \Irr(G'N/G') \), we see that \( \Van(\lambda\theta_1) = \Van(\theta_1) \) while \( \lambda\theta_1 \neq \theta_ 1\) for  some \( \lambda \neq 1 \), by Lemma \ref{newchar}. This implies that \( \theta_1 \) would be adjacent to \( \lambda\theta_1 \), leading us to the conclusion that \( \lambda\theta_1 = \theta_2 \), which is a contradiction since \( \Van(\theta_1) \neq \Van(\theta_2) \). Hence, we confirm that \( G'N = G' \).

Take \( x \in   G - G'\subseteq \Van(G)\).  There is no character $\eta\in \Irr(G|G')-\{\theta_1, \theta_2\}$ such that $\eta(x)\not =0$, as otherwise $\lambda\eta(x)\not =0$ for all $\lambda\in \Irr(G/G')$. As $\lambda\eta\not =\eta$ for some $\lambda\in \Irr(G/G')$, by Lemma \ref{newchar}, we get that $\eta$ would be adjacent to $\lambda\eta$ for some $\lambda\in \Irr(G/G')$ implying a contradiction. Consequently, all characters in \( \Irr(G|G') \) must vanish on \( G - G' \).  So $\Van(\eta)=\Van(G)$ for all $\eta\in \Irr(G|G')-\{\theta_1,\theta_2\}.$ Moreover, $G$ is a $2$-nilpotent group by  \cite[Corollary 2.2]{BDS}.

%%%%%%%%%%%%%%%%%%%%%%% 

%There can be at most one non-linear irreducible character \( \eta \in \Irr(G/T) \) for which \( \eta(x) \neq 0 \); if there were more, it would imply two edges in \( \Gamma_{nv}(G) \), which is impossible. Moreover, applying the second orthogonality relation suggests that having such a non-linear irreducible character is not feasible. It derives from \( \sum_{\lambda \in \Irr(G/G')} \lambda(x) + \eta(x)\eta(1) = 0 \), leading to \( \eta(x) = 0 \).

  Note that $N$ is not trivial as otherwise $\Gamma_{nv}(G)$ is null. Let's consider \( N/T \) as a chief factor of \( G \). Therefore,  \( \Irr(G/T) \) consists of at least three non-linear irreducible characters.
%%%%%%%%%%%%%%%%%%%%%%%%%%%%%
 If  \( N/T \) is not a  \( 2\)-group, then by Lemma \ref{extraspecial} we would then demonstrate that \( G/K \), the Sylow $2$-subgroup of $G/T$,  acts Frobeniusly on \( K/T \) and $K/T$ is abelian. Thus $K/T\subseteq \Nv(G/T)$,  applying  \cite[Corollary 2.2]{BDS}. If  $N/T$ is a $2$-group. As $G$ and $G/T$ is $2$-nilpotent, we have that $N/T\subseteq {\bf Z}(G/T)$. Thus applying Lemma \ref{pq}, we have $K/T\subseteq \Nv(G/T)$. Hence in both cases, we conclude  that $K\subseteq \Nv(G)$ (as otherwise  $\Irr(G/T|G'/T)$ forms a complete subgraph of $\Gamma_{nv}(G)$ with  at least three vertices). As $\Van(\theta_2)=G-K$
 we get that $\theta_2$ is an isolated vertex of $\Gamma_{nv}(G)$, a contradiction. 
 %%%%%%%%%%%%%%%%%%%%%%%%%%%%%%%%
  \end{proof}
 %%%%%%%%%%%%%%%%%%%%%%%%%%%%
  %This contradicts our earlier statement that \( 9=|C_{G/T}(gT)| =|K/T| \) for each \( g \in \Van(G) \cap (K-N) \). Therefore, we must assume that \( N/T \) is not a \( 3 \)-group.

%%%%%%%%%%%%%%%%%%%%%%%%%%%%%% 
% Recall that  $K/N$ is  the Sylow 3-subgroup of $G/N$.  Consider \( g \in \Van(G) \cap (K-N) \). By the second orthogonality relation, we find that \( |C_{G/T}(gT)| = |G/T : G'T/T| + |\theta_1(gT)|^2 + |\theta_2(gT)|^2 = |G : G'| + |\theta_1(gT)|^2 + |\theta_2(gT)|^2 = 4 + (2)^2 + (-1)^2 = 9 \).
%%%%%%%%%%%%%%%%%%%%%%%%%%%%%%%%%%%%%

% If \( x \in G - G' \), all non-linear irreducible characters vanish on \( x \), implying that \( |C_{G/T}(xT)| = |G : G'| = 4 \) and if \( x \in G' - K \), all non-linear irreducible characters aside from \( \theta_1 \)  vanish on \( x \), leading to \( |C_{G/T}(xT)| = |G : G'| + |\theta_1(xT)|^2 = 4 + 4 = 8 \).

%Thus, \( G/K \) which is the Sylow \( 2 \)-subgroup of \( G/T \) acts Frobeniusly on \( K/T \). 

%Since \( |G/K| \) is even, we conclude that \( K/T \) must be abelian. 

%Consequently, given that \( |C_{G/T}(gT)| = 9 \) for each \( g \in \Van(G) \cap (K-N) \), it follows that \( C_{K/T}(N/T) \subseteq N/T \), by applying \cite[Lemma 2.9]{silvio1}. So  we conclude that \( K - N \subseteq \Van(G) \). Therefore, for any \( g \in K - N \), we find \( |C_{G/T}(gT)| = 9 \), implying that \( K/N \) acts Frobeniusly on \( N/T \), a contradiction since \( K/N\cong C_3^2 \) is not cyclic.	

   \begin{theorem}\label{triple}
   	Let $G$ be a finite group and $M $ be a normal subgroup of $G$ such that $G/M$ has exactly one non-linear irreducible character $\chi$. Assume $\chi$ is not an isolated vertex in $\Gamma_{nv}(G)$.  Then one of the following occurs:
   	
(a) $\Gamma_{nv}(G)$ has exactly one edge $\{\chi, \lambda\chi\}$ for some $1\not= \lambda\in \Irr(G/G')$ and for all $\theta\in \Irr(G|G')-\{\chi, \lambda\chi\}$, $\Van(\theta)=\Van(G)$. 
   	
   	b) $\Gamma_{nv}(G)$ has a triangle;  or
   	
   	c) $\Gamma_{nv}(G)$ contains two disjoint edges. 
   \end{theorem}
\begin{proof}
	
	On the contrary, assume $\Gamma_{nv}(G)$ does not satisfy any of the three conditions (a)-(c).  So by Theorem C, $G$ is solvable.  
By the hypothesis we have  $\Irr(G/M|G'M/M)=\{\chi\}$. Then   $\Van(\chi)=G-G'M$, by \cite{seitz}.  The proof proceed with the following steps:

\smallskip

{\bf Step 1.} All characters in  $ \Irr(G|G')-\{\chi\}$ are vanishing on $\Van(G)\cap (G'M-M)$. In particular, there exists an element $g_0\in \Van(G)\cap M$ and  unique non-linear character $\theta\in \Irr(G|M)$ such that $\theta(g_0)\neq 0$. 

\smallskip

 Assume otherwise. Then for some $\chi_0\in \Irr(G|G')-\{\chi\}$  there  is an element $g\in \Van(G)\cap (G'M-M)$ such that $\chi_0(g)\not =0$. Since $\Gamma_{nv}(G)$ is triangle-free, then this character must be the only character in  $\Irr(G|G')-\{\chi\}$ such that it is not vanishing on $g$. Thus applying the second orthogonality relation on $g$ we have  $\sum_{\theta_0\in \Irr(G/M)}\theta_0(g)\theta_0(1)+  \chi_0(1)\chi_0(g)=0$. Now again using the second orthogonality  relation on $gM$, we have that $\sum_{\theta_0\in \Irr(G/M)}\theta_0(gM)\theta_0(M)=\sum_{\theta_0\in \Irr(G/M)}\theta_0(g)\theta_0(1)=0$ implying that $\chi_0(g)=0$, a contradiction. Thus every character in $\Irr(G|G')-\{\chi\}$ is vanishing on $\Van(G)\cap (G'M-M)$, as claimed.  
As $\chi$ connects to another irreducible character, we then know for some $g_0\in \Van(G)\cap M$ there exists a unique non-linear character $\theta\in \Irr(G|M)$  such that $\theta(g_0)\neq 0$ (the reason for uniqueness is that we do not have any triangle in the graph).  
We will subsequently fix this unique character $\theta$ and the element $g_0\in \Van(G)\cap M$, where $\theta$ is adjacent to $\chi$, as $\chi(g_0)\theta(g_0)\not=0$. 

\smallskip

% \smallskip

{\bf Step  2.}  Let $N\subset M$ be a normal subgroup of $G$ such that $M/N$ is a chief factor of $G$. Then  $|\Irr(G/N|G'N/N)|\geq 3$.
\smallskip

Assume, on the contrary, that $\Irr(G/N|G'N/N)=\{\chi, \eta\}$. By our hypothesis, $G/M$ is classified as either an extraspecial $2$-group or a Frobenius group of order $p^a(p^a-1)$ with abelian Frobenius kernel of order $p^a$ and cyclic Frobenius complement, where $p$ is a prime and $a$ is an integer, as shown in \cite{seitz}. Referring to \cite{BCH92} and \cite[Theorem 7 and Lemma 2]{Ber96}, we identify three possible cases: 
 $G/N \cong C_3^2 : Q_8$; 
$G/N$ is a $2$-group with $G/N/{\bf Z}(G/N)$ being an extraspecial $2$-group, where $|G'N/N|=2$ and $|{\bf Z}(G/N)|=4$; or 
$G/N/{\bf Z}(G/N)$ is a Frobenius group of order $p^a(p^a-1)$ and ${\bf Z}(G/N)$ is a cyclic group of order $2$ for some integer $a$ and prime $p$.
In all these scenarios, we have $M/N \subseteq \Nv(G/N)$.
From Step 1, there exists a unique $\theta \in \Irr(G|G') - \{\chi\}$ such that $\theta(g_0) \neq 0$ for some $g_0 \in \Van(G) \cap M$. This means that $\eta = \theta$. In all scenarios, $G - G'M = \Van(\chi) \subseteq \Van(\eta)$. 

If there is any non-linear irreducible character $\zeta$ of $G$ that does not vanish on $g \in G - G'M$, then $\lambda \zeta(g) \neq 0$ for all $\lambda \in G/G'$, leading to the existence of two disjoint edges, by Lemma \ref{newchar}. If there is a character $\zeta \in \Irr(G|G') - \{\chi, \theta\}$ such that $\zeta(g) \neq 0$ for some $g \in \Van(G) \cap M$, then $\{\zeta, \theta, \chi\}$ forms a complete graph, resulting in a contradiction. Therefore, by Step 1, for all $\zeta \in \Irr(G|G') - \{\chi, \theta\}$, we find that $\Van(\zeta) = \Van(G)$. 

As a result, in this case, $\Gamma_{nv}(G)$ has exactly one edge $\{\chi, \theta = \eta\}$. If $G/N\cong C_3^2:Q_8$, then we get a contradiction by Lemma \ref{quaterniun}. 
In all the other  possible cases for $G/N$, we find that $\theta = \lambda \chi$, a contradiction, as we assumed $G$ is a counterexample to the theorem.    

%	Now, consider the case where $G/N \cong \mathbb{Z}_3^2\rtimes Q_8$. Here, we observe that $\Van(\chi) \subseteq \Van(\eta) = G - M$. Hence, we have $\Van(\eta) \cup \Van(\theta_0) = \Van(G)$ for all $\theta_0 \in \Irr(G|N)$, indicating that the only edge in $\Gamma_{nv}(G)$ is $\{\eta, \chi\}$. According to Theorem C, this would imply that $G \cong S_4$, leading to a contradiction.

%	Next, we explore the situation where $G/N$ fits one of the other two possibilities. In these cases, we find that  $\Van(\chi) = \Van(\eta) = G - G'N=G-G'M$. In this scenarios,  we again observe that $\theta_0(h)=0$ for all $h\in G'M\cap \Van(G)$ as otherwise $\{\theta_0, \chi, \eta\}$ forms a triangle. Thus 	
% $\Van(\chi) \cup \Van(\theta_0) = \Van(G)$ for all $\theta_0 \in \Irr(G|N)$. This infers that the only connection in $\Gamma_{nv}(G)$ remains as $\{\eta, \chi\}$, which again leads us to conclude, by Theorem C, that $G \cong S_4$, resulting in another contradiction.
\smallskip

{\bf Step 3. } Final contradiction.

\smallskip 

% We use the following claims several times. So we bold them as claims.  

%{\bf Claim 1.} \( \Van(G) \subseteq G - N \)

It is  important to note that \( \Van(G) \cap N = \emptyset \). If that is not the case, then \( \Irr(G/N|G'N/N) \)  leads to a complete subgraph in \( \Gamma_{nv}(G) \).  As $|\Irr(G/N|G'N/N)|\geq 3$ we get a contradiction. Thus, we conclude that \( \Van(G) \subseteq G - N \). 

% {\bf Claim 2.} 

%We claim  that \( G'M = G' \) and consequently $\Van(\chi)=G-G'$. 
%If this is not true, there  exists some \( \lambda \in \Irr(G'M/G') \) such that \( \Van(\lambda\chi) = \Van(\chi) \). This would also imply that the set \(  \chi\) and \( \lambda\chi \) have the same neighbors in $\Gamma_{nv}(G)$, which again results in a contradiction. It is crucial to remember that \( \Van(\chi) = G - G' \).

% {\bf Claim 3.} \( G'/N \not\cong M/N \times G'/M \).   

% On the contrary assume \( G'/N = M/N \times T/N \) with \( T/N \cong G'/M \). Here, for every non-linear character \( \theta_0 \in \Irr(G/T) \cup \Irr(G/M) \), it follows that \( G' - N \subseteq \Nv(\theta_0) \). Consequently, we have \( \Van(\theta_0)\subseteq \Van(\chi) =  G - G' \) for all non-linear characters \( \theta_0 \in \Irr(G/T) \). Therefore, the set \( \{\theta, \chi, \theta_0\} \) would form a triangle for all non-linear characters \( \theta_0 \in \Irr(G/T) \), leading to  another contradiction. Hence, we conclude that \( G'/N \not\cong M/N \times G'/M \).   

{\bf  Case 1.} Assume  that $G/M$ is an extraspecial $2$-group. 

Therefore, either  $G'M/N \cong M/N\rtimes G'M/M\cong  C_p^a \rtimes C_2$ (can be a direct product), for some odd prime $p$, or $|G'M/N|=4$. If we consider the second scenario, we find that $M/N \subseteq \mathbf{Z}(G/N)$. In the first case, the Sylow $p$-subgroup of $G/N$ is both normal and abelian, implying that no character degree in $\Irr(G/N)$ is divided by $p$ and so using \cite[Corollary 2.2]{BDS},  in both situations—regardless of whether $M/N$ is a $2$-group or not—it follows that $M/N \subseteq \Nv(G/N)$. This leads to the conclusion that $M \cap \Van(G) = \emptyset$, as having any overlap would mean that $\Irr(G/N|G'N/N)$ forms a complete graph, contradicting Step 1. 
% indicating that $g_0$ is located in $G' - M$. If there exists any other non-linear irreducible character aside from $\theta$ and $\chi$ that does not vanish on $g_0$, a triangle would form in $\Gamma_{nv}(G)$, which is a contradiction. Therefore, applying the second orthogonality relation, we find that:

%  $$\sum_{\lambda \in \Irr(G/G')} \lambda(g_0) + \chi(g_0) \chi(1) + \theta(g_0) \theta(1) = 0.$$

% Since $\sum_{\lambda \in \Irr(G/G')} \lambda(g_0) + \chi(g_0) \chi(1) = \sum_{\theta_0 \in \Irr(G/M)} \theta_0(g_0) \theta_0(1) = 0$, it follows that $\theta(g_0) = 0$, which leads to a contradiction.

% If there is a character in $\chi\not =\theta_0\in \Irr(G/N|G'/N)$ such that $G-G'\subseteq \Van(\theta_0)$, then either $\Van(\theta_0)=G-M=\Van(G)$ ($|G'/M|=2$) and so we have an isolated vertex which is impossible; or $\Van(\theta_0)=G-G'$ implying that $\{\theta, \theta_0, \chi\}$  forms a triangle and again this is impossible. So for each $\theta_0\in $ 

%Thus we may assume $G'/M=C_p^{a}\rtimes C_2$ and in this case $G'/M$ is a Frobenius group. Note that  for all characters $\theta\in \Irr(G/N|G'/N)$, we have  $M/N\subseteq \Nv(\theta)$ and for all characters $\theta\in \Irr(G/N|M/N)$ we have $G'-M\subseteq \Van(\theta)$. If $M \subseteq \cap \Van(G)\not=\emptyset$, then  $\Irr(G/N|G'N/N)$ forms a complete graph wich is a contradiction. So $M\subseteq \Nv(G)$. Thus $\chi$ must be an isolated vertex in $\Gamma_{nv}(G)$ and this is also a contradiction. 

{\bf Case 2.} Assume $G/M$ is a Frobenius group of order $p^a(p^a-1)$ with abelian kernel of order $p^a$ and cyclic complement of order $p^a-1$.

Initially we show that $G'M=G'$. First assume $|G'M/G'|>2$. Then for all $\lambda\in G'M/G'$ we get that $\Van(\chi)=\Van(\lambda\chi)$, implying that $\lambda\chi$ is adjacent to $\chi$ and in addition $\lambda\chi$
has the same neighbors as $\chi$ for all $1\not=\lambda\in \Irr(G'M/G')$. As $G'M\cap \Van(\chi)=\emptyset$, we get that $\lambda\chi\not =\chi$ for each $\lambda\in  \Irr(G'M/M)$ leading to existence of a triangle, a contradiction.   So we may assume  $|G'M/G'|=2$.  Therefore the group $M/(G'\cap M)\cong C_2$ is a cheif factor of $G$. Setting $N_0:=G'\cap M$ we get that $G/N_0$ has at least three non-linear irreducible character by Step 2.  But $G/N_0$ has a normal subgroup $M/N_0$ of order 2, implying that $M/N_0={\bf Z}(G/N_0)$ and $G/N_0/{\bf Z}(G/N_0)$ is a Frobenius group of order $p^a(p^a-1)$. This group has exactly two non-linear irreducible characters by \cite[Theorem 7]{Ber96},  contradicting Step 2.  So we have $G'M=G'$, as desired.  

Now, we will demonstrate that for every non-linear character $\theta_0 \in \Irr(G|G')$, the character $\theta_0$ vanishes on the set $G - G'$. Suppose there exists a non-linear irreducible character  $\theta_0$ that does not vanish on some element $g \in G - G'$.  Then $\lambda\theta_0\not =\theta_0$ and  $\Van(\lambda\theta_0)= \Van(\theta_0)$, for  some  $\lambda\in \Irr(G/G')$, by Lemma \ref{newchar}.  If $\theta_0(h)\not =0$ for some $h\in M\cap \Van(G)$, then $\{\chi, \theta_0, \lambda\theta_0\}$ forms a complete graph. Thus $\theta_0$   vanishes on $M\cap \Van(G)$, implying that $\theta\not \in \{\theta_0,\lambda\theta_0\}$.  Thus $\{\chi, \theta\}$ and $\{\theta_0, \lambda\theta_0\}$ for some $\lambda\in \Irr(G/G')$, are two disjoint edges of $\Gamma_{nv}(G)$, contradicting our assumptions. 
%%%
%%%%
%  This implies that $\theta_0 \in \Irr(G/N | \lambda)$, where $\lambda \in \Irr(G' / N | M / N)$ and the index $|I_G(\lambda) / G'| > 1$. Consequently, the character $\lambda$ has $t = |I_G(\lambda) / G'|$ extensions, denoted as $\{\eta_1, \dots, \eta_t\}$, to the group $I_G(\lambda)$. Each of these $\eta_i^G$ belongs to $\Irr(G)$, leading us to conclude that $\theta_0 = \eta_i^G$ for some $1 \leq i \leq t$. 
%%%%% 
%  Moreover, if $\eta_i^G(g) \neq 0$ for a specific $g$, then it follows that $\eta_j^G(g) \neq 0$ for all $j \in \{1, \dots, t\}$. This indicates that all these $t$ characters must be adjacent, resulting in the conclusion that $t = 2$. We can then assume $\theta_0 = \eta_1^G$. 
%%%%%%%%%%%%
% Now, if $\eta_1^G(h) \neq 0$ for some $h \in G' \cap \mathrm{Van}(G)$, then $\eta_2^G(h) \neq 0$ as well, meaning that the set $\{\chi, \eta_1^G, \eta_2^G\}$ forms a triangle, which leads to a contradiction.  
%%%
%%%
%As a result, we can conclude that all other characters in $\Irr(G|G') - \{\chi\}$, (even $\theta$), must vanish on $g_0\in M\cap \Van(G)\subseteq G' \cap \mathrm{Van}(G)$, leading to a contradiction.
Hence we have  all characters in $\Irr(G|G')$ are vanishing on $G-G'$, as desired.

  Thus by the second orthogonality relation we have $|C_{G}(g)|=|G:G'|$ for all $g\in G-G'$. By applying Lemma \ref{Frobenius}, we get that $G$ is a Frobenius group whose kernel is $G'$.   
% Next, we aim to show that \( (|\overline{G'}|, |\overline{G}/\overline{G'}|) = 1 \). Assume for contradiction that a prime \( q \) divides both \( |\overline{G'}| \) and \( |\overline{G}/\overline{G'}| \). Let \( Q \) be a Sylow \( q \)-subgroup of \( \overline{G} \). The intersection \( Q \cap \overline{G'} \) forms a normal subgroup of \( Q \), consequently leading to \( Q \cap \overline{G'} \cap \mathbf{Z}(Q) \neq 1 \). Choosing an element \( h \in Q - \overline{G'} \), we conclude that \( 1 \neq Q \cap \overline{G'} \cap \mathbf{Z}(Q) \subset C_{\overline{G}}(h) \). This result contradicts our earlier statement because \( |Q\overline{G'}/\overline{G'}| \cdot |Q \cap \overline{G'} \cap \mathbf{Z}(Q)| \) must divide \( |C_G(h)| = |\overline{G}/\overline{G'}| \), given that \( Q \overline{G'}/\overline{G'} \) is cyclic.   
%This  implies that $G/G'$ acts  Frobeniously on $G'/N$. 
Thus $G'$ is nilpotent. Hence $G/N$ is a Frobenius group with Frobenius kernel  $G'/N$.    

Recall that $\Van(G)\subseteq G-N$. Now we show that $M/N\subseteq \Nv(G/N)$.  If $p$ does not divide $|M/N|$, then $M/N$ is a normal and abelian Sylow subgroup of $G/N$ and all irreducible character degrees of $G/N$ are coprime to $|M/N|$, implying that $M/N\subseteq \Nv(G/N)$, by \cite[Corollary 2.2]{BDS}, as wanted. If  $M/N$ is a $p$-group, then  applying Lemma \ref{pp}, we get that $M/N\subseteq \Nv(G/N)$, as desired.  
 As \(g_0\in  \Van(G) \cap M \neq \emptyset \), it indicates that \( \Irr(G/N|G'/N) \) forms a complete graph, which is a final contradiction.	
	\end{proof}
\begin{lemma}\label{pnil}
	Let $G$ be a finite  group whose common non-zero graph has exactly one edge $\{\theta, \zeta\theta\}$, for some $\zeta\in \Irr(G/G')$ and $\theta\in \Irr(G|G')$. If $\Van(\chi)=\Van(G)$ for  all $\chi\in \Irr(G|G')-\{\theta, \zeta\theta\}$, then there is a prime $ p $ such that $ p\mid \chi(1) $ for all $ \chi\in \Irr(G|G') $. In particular, $ G $ is $ p $-nilpotent for some prime $ p $.  
\end{lemma}
\begin{proof}
	\cite[Theorem B]{MNO} guarantee that there is  an $p$-element $ g\in \Van(\theta) =\Van(\zeta\theta)$ for some prime $ p $. Then $ p\mid \chi(1) $  for all $\chi\in \Irr(G|G')$, by \cite[Corollary 2.2]{BDS} and   the result follows.
	% or $ p\nmid \theta_{2}(1) $. If $ p\mid \theta_{2}(1) $, then the result follows. So we may assume that $ p\nmid \theta_{2}(1) $. This means that there is exactly one nonlinear character of $ G $ not divisible by $ p $. By Berkovich, $ G/G' $ is of order $ p^{n}-1 $. Let $ h\in \Van(\theta_2) $ be a $ q $-element $ p\not= q $. If $ q\mid \theta_{1}(1) $, then the result follows. If $ q\nmid \theta_{1}(1) $, then by Berkovich $ G/G' $ is order $ q^{n}-1 $.
\end{proof}

 \begin{theorem}\label{key3}
	Let $G$ be a finite group and $\Gamma_{nv}(G) $ has exactly one edge $\{\theta_1, \theta_2\}$.  If $\ker\theta_1=\ker\theta_2$, then $G/\ker\theta_1$ has at least three non-linear irreducible characters. 
\end{theorem}
\begin{proof} Note that by Theorem C, $G$ is solvable.  
	 Assume that $\ker\theta_1=\ker\theta_2$, and consider that $G/\ker\theta_1$ has exactly two non-linear irreducible characters, with both $\theta_1$ and $\theta_2$ being faithful characters of $G/\ker\theta_1$. According to \cite[Theorem 7 and Lemma 2]{Ber96} and \cite{BCH92}, this implies that $G$ is either an extraspecial 3-group; or a Frobenius group of order $p^a(p^a-1)/2$, with abelian Frobenius  kernel of order $p^a$ and cyclic Frobenius complement,  for some odd prime $p$ and integer $a$. In both cases $\Van(\theta_1)=\Van(\theta_2)=G-G'\ker\theta_1\subseteq \Van(G)$.
	 \smallskip 
	 
	 {\bf Step 1.} For all $\chi\in \Irr(G|G')-\{\theta_1, \theta_2\}$,  $\Van(\chi)=\Van(G)$. 
	 
	 \smallskip

	 Take any non-linear character $\chi\in \Irr(G)$ such that $\chi(g)\neq0$ for some $g\in G-G'\ker\theta_1$. Then, it follows that $\zeta\chi(g)\neq0$ for all $\zeta\in \Irr(G/G')$, leading to the conclusion that $\chi$ is adjacent to $\zeta\chi$, for some $\zeta\in \Irr(G/G')$, by Lemma \ref{newchar}. This presents a contradiction. Consequently, we can deduce that $G-G'\ker\theta_1\subseteq \Van(\chi)$ for all non-linear irreducible characters $\chi$. We consider an element $g\in \Van(G)\cap G'\ker\theta_1$. For any non-linear irreducible character $\chi\in \Irr(G)- \{\theta_1, \theta_2\}$, we find that $\chi(g)=0$. If this were not true, it would imply that $\chi$ is adjacent to $\theta_1$, which leads to another contradiction. Thus $\Van(\chi)=\Van(G)$ for all $\chi\in \Irr(G)-\{\theta_1, \theta_2\}$. 
	
	\smallskip
	
	{\bf Step 2. } $G'=G'\ker\theta_1$  and so  all characters in $\Irr(G|G')$ vanish on $G-G'$.
	
	\smallskip
	
	We observe that $G' \subseteq G'\ker\theta_1$. If we assume $G' \neq G'\ker\theta_1$, we can select $\lambda \in \Irr(G'\ker\theta_1/G')$ such that $\lambda\theta_1 \neq \theta_1$,  by Lemma \ref{newchar}, and $\Van(\lambda\theta_1)=\Van(\theta_1)=\Van(\theta_2)$. This leads to the conclusion that $\theta_2=\lambda\theta_1$, which implies $\ker\theta_2=\ker\theta_1\subseteq \ker\lambda\theta_1$. Therefore, we deduce that $\ker\theta_1\subseteq \ker\lambda$, resulting in a contradiction. Hence, we affirm that $G'=G'\ker\theta_1$.  Thus by Step 1,  we get that all characters in $\Irr(G|G')$ vanish on $G-G'$.
	
%	\smallskip
	
	% {\bf Claim 2. } 
	
	% Now, take any non-linear character $\chi\in \Irr(G)$ such that $\chi(g)\neq0$ for some $g\in G-G'$. Then, it follows that $\zeta\chi(g)\neq0$ for all $\zeta\in \Irr(G/G')$, leading to the conclusion that $\chi$ is adjacent to $\zeta\chi$. This presents a contradiction. %Additionally, should there exist characters $\chi_1$ and $\chi_2$ such that $\chi_1(g)\chi_2(g) \neq 0$ for some $g \in G-G'$, it would indicate that $\chi_1$ is adjacent to $\chi_2$, which is impossible. Thus, we conclude that there is at most one character $\chi_0 \in \Irr(G|G')$ for which $\chi_0(g) \neq 0$ for some $g \in G-G'$. Therefore, by the second orthogonality relation, we can infer that $\sum_{\chi\in \Irr(G/G')}\chi(g)+\chi_0(1)\chi_0(g)=0$, leading us to conclude that $\chi_0(g)=0$, which is a contradiction.
	%	  As a result, all characters in $\Irr(G|G')$ vanish on $G-G'$.
	% 	Thus $\chi(g)=0$ for all $\chi\in \Irr(G|G')$ and $g\in (G-G')$. 
	
	\smallskip
	
	{\bf Step 3.} There is a subgroup  $N\subset \ker\theta_1$, such that $\ker\theta_1/N$ is a chief factor of $G$. Setting $\overline{G}=G/N$, we have that $\Van(\overline{G})\cap \overline{G'}\not=\emptyset$. Moreover, if $\overline{\ker\theta_1}\cap \Van(\overline{G})=\emptyset$, then $\Van(G)\cap \ker\theta_1=\emptyset$. 
	
	\smallskip
	
	Note that $\ker\theta_1\not =1$ as otherwise $\Van(G)=\Van(\theta_1)=\Van(\theta_2)$ and so $\theta_1$ and $\theta_2$ are isolated vertices. Thus  there is a subgroup  $N\subset \ker\theta_1$, such that $\ker\theta_1/N$ is a chief factor of $G$.  Note that $\Irr(\overline{G})$ has at least three non-linear irreducible characters. If $\Van(\overline{G})\cap \overline{G'}=\emptyset$, then
	for all $\chi\in \Irr(\overline{G}|\overline{\ker\theta_1})$ we get that $\Van(\chi)=\Van(G)=G-G'$, implying that $\Gamma_{nv}(G)$ is null, a contradiction. 
	  So we assume $\Van(\overline{G})\cap \overline{G'}\not =\emptyset$. Moreover, if $\Van(\overline{G})\cap \overline{\ker\theta_1}=\emptyset$ while $\Van(G)\cap \ker\theta_1\not =\emptyset$, we get that all characters in $\Irr(\overline{G}|\overline{G'})$ forms a complete graph, which again leads to another contradiction. 
	\smallskip
	
	{\bf Step 4. }  $G/\ker\theta_1$ is not a Frobenius group of order  \( \frac{p^a(p^a-1)}{2} \), with abelian Frobenius kernel of order $p^a$ and cyclic Frobenius complement,  for some odd prime $p$. 
	
	\smallskip
	
	Assume that \( G/\ker\theta_1 \) is a Frobenius group with  order  \( \frac{p^a(p^a-1)}{2} \) for some odd prime $p$, where the Frobenius complement is cyclic. By applying the second orthogonality relation and referring back to Steps 1 and 2, we find that for every \( x \) in \( \overline{G} - \overline{G'} \), the equality \( |\overline{G}/\overline{G'}| = \sum_{\chi \in \Irr(\overline{G}/\overline{G'})} |\chi(x)|^2 = |C_{\overline{G}}(x)| \) holds. Thus applying Lemma \ref{Frobenius} we get that $\overline{G}$ is a Frobenius group with Frobenius kernel $\overline{G'}$.  
	%%%%%%%%%%%%%%%%%%%
	% Next, we aim to show that \( (|\overline{G'}|, |\overline{G}/\overline{G'}|) = 1 \). Assume for contradiction that a prime \( q \) divides both \( |\overline{G'}| \) and \( |\overline{G}/\overline{G'}| \). Let \( Q \) be a Sylow \( q \)-subgroup of \( \overline{G} \). The intersection \( Q \cap \overline{G'} \) forms a normal subgroup of \( Q \), consequently leading to \( Q \cap \overline{G'} \cap \mathbf{Z}(Q) \neq 1 \). Choosing an element \( h \in Q - \overline{G'} \), we conclude that \( 1 \neq Q \cap \overline{G'} \cap \mathbf{Z}(Q) \subset C_{\overline{G}}(h) \). This result contradicts our earlier statement because \( |Q\overline{G'}/\overline{G'}| \cdot |Q \cap \overline{G'} \cap \mathbf{Z}(Q)| \) must divide \( |C_G(h)| = |\overline{G}/\overline{G'}| \), given that \( Q \overline{G'}/\overline{G'} \) is cyclic.
	%%%%%%%%%%%
	% this implies that $|C_G(h^{O(h)/q})|=|G/L|$ and so  $(|G/L|,|L|)=1$
	%%%%%%%%%%%5
	This indicates that $\overline{G'}$ is a nilpotent group.  
		
	If $\overline{G'}$ is not a $p$-group, then $\overline{G'}\cong \overline{\ker\theta_1}\times \overline{H}$ where $\overline{H}$ is the sylow $p$-subgroup of $\overline{G'}$ isomorphic to $ G'/\ker\theta_1$.  Then $\overline{G}/\overline{H}$  is a Frobenius group whose Frobenius complement is isomorphic to $\overline{G}/\overline{G'}\overline{H}=\overline{G}/\overline{G'}$ and so for all  non-linear characters  $\chi\in \Irr(\overline{G}/\overline{H})$, $\Van(\chi)=G-G'=\Van(\theta_1)=\Van(\theta_2)$, by Lemma  \ref{pq}, which is a contradiction. Thus $\overline{G'}$ is a $p$-group.    	 	   	 
	If we assume $\overline{G'}$ is abelian, using Lemma \ref{pq} and  Step 3 we get a contradiction. 
	
	Therefore, we can affirm that $\overline{G'}$ is a non-abelian $p$-group, leading us to the conclusion that $\overline{\ker\theta_1}={\bf Z}(\overline{G'})=\overline{G'}$. Assume $|\overline{\ker\theta_1}|=p^b$ for some integer $b$.  	
	 By Step 3 and applying Lemma \ref{pp}, 	we 
	 %Next, we demonstrate that $\overline{\ker\theta_1}\cap \Van(\overline{G})=\emptyset$. Assume the opposite to be true. Then we have some $g\in \overline{\ker\theta_1}\cap \Van(\overline{G})$. This implies that for any $\chi\in \Irr(\overline{G}|\overline{\ker\theta_1})$, it holds that $\chi(g)=0$. In fact, if this is  not the case, the set $\{\theta_1, \theta_2, \chi\}$ would form a triangle for each $\chi\in \Irr(\overline{G}|\overline{\ker\theta_1})$, which is clearly a contradiction. By using the second orthogonality relation, we then obtain $|G:G'|+\theta_1(1)\theta_1(g)+\theta_2(1)\theta_2(g)=|G:G'|+\theta_1(1)^2+\theta_2(1)^2=3|G:G'|=0$, resulting in another contradiction. Thus  $\overline{\ker\theta_1}\cap \Van(\overline{G})=\emptyset$.
	%%%%%%%%%%%%%%%%%%%%%%%%%
%	Take an element $g\in \Van(\overline{G})\cap (\overline{G'}-\overline{\ker\theta_1})$. On one hand, by Step 1, we have $$p^a=|C_{\overline{G}/\overline{\ker\theta_1}}(g\overline{\ker\theta_1})|=|G:G'|+ |\theta_1(g)|^2+|\theta_2(g)|^2,  $$ which means that $|\theta_1(g)|^2+|\theta_2(g)|^2=(p^a+1)/2$. On the other hand, it follows that $$p^b=|\overline{\ker\theta_1}|<|C_{\overline{G}}(g)|=|G:G'|+ |\theta_1(g)|^2+|\theta_2(g)|^2= (p^a-1)/2 +(p^a+1)/2=p^a,$$ which
 infer that $1\leq b<a$. Given that $\overline{G}/\overline{G'}$ acts Frobeniusly on $\overline{\ker\theta_1}$, we find that $|\overline{G}/\overline{G'}|=(p^a-1)/2$ divides $|\overline{\ker\theta_1}|-1=p^b-1$. This indicates that $(p^a-1)/2=p^b-1$ and so $(p^a-1)$  lacks a primitive prime divisor, leading us to the possibilities: either $a=1$; $a=2$ with $p+1$ being a power of $2$; or $a=6$ with $p=2$, according to Zsigmondy's Theorem \cite[Theorem V]{zsig}. The first scenario is untenable since $1\leq b<a=1$. The third scenario also fails as $2$ divides $p^a-1$. Thus, the only feasible case is $a=2$ and $b=1$. This scenario implies that $(p^2-1)/2=p-1$, leading to the conclusion that $(p+1)/2=1$, which is a contradiction.
	\smallskip

	{\bf Step 5.}  Final contradiction. 
	\smallskip
	
	By Step 4, we assume that $G/\ker\theta_1$ is an extraspecial 3-group. We begin by showing that $\overline{\ker\theta_1}\cap \Van(\overline{G})=\emptyset$. Let $|\overline{\ker\theta_1}|=q^b$, where $q$ is a prime different from 3, and $b$ is an integer. Since for every $\chi\in \Irr(\overline{G}|\overline{G'})$, the value $\chi(1)$ is not divisible by $q$, we can conclude that $\overline{\ker\theta_1}\cap \Van(\overline{G})=\emptyset$, by \cite[Corollary 2.2]{BDS}.
	If  $\overline{\ker\theta_1}$ is a 3-group, then it follows that $\overline{\ker\theta_1}\subseteq {\bf Z}(\overline{G})$, thus confirming our claim. 
	
	 Consequently, we establish that $\Van(G)\cap \ker\theta_1=\emptyset$, by Step 3. But according to Step 3, we have $\Van(\overline{G})\cap \overline{G'}\not=\emptyset$.  Thus there is some element $g\in \Van(G)\cap (G'-\ker\theta_1)$.  If $|\overline{\ker\theta_1}|=q^b$ for some prime $q\neq 3$ and integer $b$, then applying Lemma \ref{extraspecial} and considering Step 1 and Step 2, we get a contradiction. Thus $q=3$. 
	%
%	If  $\Van(G)=G-G'=\Van(\theta_1)=\Van(\theta_2)$, then   $\theta_1$ and $\theta_2$ are isolated vertices in $\Gamma_{nv}(G)$, leading to another contradiction. Thus, we can safely assume $G-G'\subset \Van(G)\subseteq G-\ker\theta_1$. Consequently, for every $\chi\in \Irr(\overline{G}|\overline{\ker\theta_1})$, it follows that $\chi(g)=0$ for all $g\in \Van(G)\cap  (G'-\ker\theta_1)$, by the discussion proceeding the first step.  %Therefore, whether  $\overline{\ker\theta_1}$ is a $3$ group or not, by Claim 1, we have that   for every $\chi\in \Irr(\overline{G}|\overline{\ker\theta_1})$, it follows that $\chi(g)=0$ for all $g\in \overline{G}-\overline{\ker\theta_1}$
%	
Then for every  $\chi\in \Irr(\overline{G}|\overline{\ker\theta_1})$ we get that $\chi(g)=0$ for some $g\in \overline{G'}-\overline{\ker\theta_1}$. As $|\overline{G'}/\overline{\ker\theta_1}|=3$ we get that  $\overline{\ker\theta_1}={\bf Z}(\overline{G})$ and  so $G-\ker\theta_1\subseteq \Van(G)$. Thus  for every  $\chi\in \Irr(\overline{G}|\overline{\ker\theta_1})$ we get that $\chi(g)=0$ for every $g\in \overline{G}-\overline{\ker\theta_1}$, by Step 1. Now we derive a contradiction, applying Lemma \ref{pgroup}. \end{proof}
   {\bf The Proof of Theorem B. }
   	According to Theorem C, we know that $G$ is solvable. If $\theta_2=\zeta\theta_1$ for some $\zeta\in \Irr(G/G')$ and $\Van(G)=\Van(\chi)$ for all $\chi\in \Irr(G|G')-\{\theta_1, \theta_2\}$, then using Lemma \ref{pnil},  the conditions of Case 1 of Theorem B hold.  So we may  assume  that either $\theta_2 \neq \zeta \theta_1$ for any $\zeta \in \Irr(G/G')$, or there exists a character $\chi \in \Irr(G|G') - \{\theta_1, \theta_2\}$ such that $\Van(\chi) \neq \Van(G)$. In this situation, we will demonstrate that $G \cong S_4$. We will establish our findings by examining the following two cases:
  
   	{\bf Case 1.} Let $\ker\theta_1 \neq \ker\theta_2$.   We may assume  $|\ker\theta_2|\geq|\ker\theta_1|$.  Then we work on $G/\ker\theta_2$. 
  	  As $\theta_1\not \in \Irr(G/\ker\theta_2)$ we have  $G/\ker\theta_2$ is either a strongly $\mathcal{H}_1'$-group or an isomorphic to a Frobenius group, where the Frobenius kernel is an abelian group and the  Frobenius complements are $Q_8$. 
   	    	
   	Initially, we consider the latter scenario. Define $N = \ker\theta_2$. Thus, we have $G/N \cong K/N \rtimes H/N$, where $H$ and $K$ are subgroups of $G$, with $K/N \cong A$ being an abelian group and $H/N \cong Q_8$. It is important to note that $G/N$ has  a faithful non-linear irreducible character $\theta_2$ along with at least one another non-linear irreducible character, say $\chi$. Here, we identify $\Van(\theta_2) = G/N - K/N$ and $\Van(\chi) = G/N - G'N/N \subset G/N - K/N$. 
   	
   	We start by assuming that \( K \subseteq \Nv(G) \). This leads to the contradiction that \( \Van(\theta_2) = \Van(G) \). Therefore, we can conclude that \( K \cap \Van(G) \neq \emptyset \). 
   	%%%%%%%%%%%%%%%%%%%
   	%If \( G/N \) has more than two non-linear irreducible characters, then the set \( \Irr(G/N|G'N/N) \) must form a triangle in \( \Gamma_{nv}(G) \),  since \( K \cap \Van(G) \neq \emptyset \), creating  a contradiction. Thus, we find that \( |\Irr(G/N|G'N/N)| = 2 \), which implies that \( G/N \cong C_3^2:Q_8 \) according to \cite{Zhang}.
   	 Since \( \theta_2 \) is adjacent to \( \chi \), as $K\cap \Van(G)\not=\emptyset$, it follows that \( \chi = \theta_1 \),  a contradiction as we assumed that $|\ker\theta_2|\geq|\ker\theta_1|$ and consequently $\theta_1\not \in \Irr(G/N)$.

   	  Therefore, we can conclude that the first case must hold, meaning $G/\ker\theta_2$, is indeed a strongly $\mathcal{H}_1'$-group. 
   	%%%%%%%%%%%%%%%%%
   	Now,  assume $G/\ker\theta_2$ has more than one non-linear irreducible character. For each non-linear character $\chi \in \Irr(G/\ker\theta_2)$, we find that $\Van(\chi) = \Van(\theta_2) = \Van(G/\ker\theta_2)$. This leads us to discover that $\theta_1$  is adjacent to all non-linear characters in $\Irr(G/\ker\theta_2)$, which again results in a contradiction. Hence,  $G/\ker\theta_2$ can only have one non-linear irreducible character.  Thus by applying Theorem \ref{triple} and Lemma \ref{pnil} we get that  the conditions outlined in Case 1 of Theorem B holds, contradicting our assumption. 

   	{\bf Case 2. } Now we assume that $\ker\theta_1=\ker\theta_2$ and by Theorem \ref{key3} we have  $G/\ker\theta_1$ contains at least three non-linear irreducible characters, with both $\theta_1$ and $\theta_2$ being faithful characters of $G/\ker\theta_1$.
   	 Let $ N/\ker\theta_1$ represent the largest normal subgroup of $G/\ker\theta_1$ such that $G/N$ remains non-abelian. Importantly, $N/\ker\theta_1$ cannot be trivial; if it is, then $\Van(\theta_1)=\Van(\chi)$ for every non-linear character $\chi\in \Irr(G/\ker\theta_1)$, by \cite[Theorem 12.3]{Isaacs}, which would imply that $\theta_1$ is linked to at least two other non-linear characters in $\Irr(G/\ker\theta_1)$, leading to a contradiction.
   	
   	Consequently, $G/N$ must possess at least one non-linear irreducible character and $\Irr(G/N)\cap \{\theta_1, \theta_2\}=\emptyset$. For any non-linear irreducible character $\chi$ within $G/N$, we establish that $\Van(\chi)=G-NG'$.  We achieve the results by adhering the following steps:
   	%%%%%%%%%%%%%%%%%%%%%%%%%%%%%%%%%%%%%%%%%%%%%%
   	% and so  they are not vanishing on $\Van(G)\cap NG'$, a contradiction.   Letting $\Irr(\overline{G}/\overline{N}|\overline{G'N}/\overline{N})=\{\chi\}$, we have $\chi$  is distinct from $\theta_1$ and $\theta_2$. 
   	%%%%%%%%%%%%%%%%%%%%%%%%%%%%%%%%%%%%%%%%
   	
   	\smallskip
   	{\bf Step 1. }  $\Van(G)\cap G'N\subseteq \Van(\eta)$ for all $\eta\in \Irr(G|G')-|\Irr(G/N|G'N/N)|$.  In particular there is $g\in G-G'N$ such that $\theta_1(g)\theta_2(g)\not=0$.
   	\smallskip
   	
   	   	If there is an element \( g  \in  NG' \cap \Van(G) \) such that \( \eta(g) \neq 0 \) for some non-linear character \( \eta \) in \( \Irr(G) -\Irr(G/N)\), then \( \eta \) must be adjacent to \( \chi \), which leads to a contradiction. Therefore, there must be an element \( g \) in \( G - NG' \) such that \( \theta_1(g) \theta_2(g) \neq 0 \).
   	%%%%%%%%%%%%%%%%%
   	%So there exists $g\in \Van(\chi) $ such that $\theta_i(g)\not=0$ for $i=1,2$.
   	\smallskip
   	
   	{\bf Step 2.} $G-G'N\subseteq \Van(\eta)$ for every $\eta\in \Irr(G|G')-\{\theta_1, \theta_2\}$.  In particular $\Van(\eta)=\Van(G)$ for all $\eta\in \Irr(G|G')-(\{ \theta_1, \theta_2\}\cup \Irr(G/N|G'N/N))$. 
   	\smallskip

   	 Now, take any non-linear character $\eta\in \Irr(G)-\{\theta_1, \theta_2\}$ such that $\eta(g)\neq0$ for some $g\in G-G'N$. Then, it follows that $\zeta\eta(g)\neq0$ for all $\zeta\in \Irr(G/G')$, leading to the conclusion that $\eta$ is adjacent to $\zeta\eta$, for some $\zeta\in \Irr(G/G')$ by Lemma \ref{newchar}. This presents a contradiction.
   	%%%%%%%
   	%%%%%%5
   	%If \( g \in \Nv(\gamma) \cap \Nv(\eta) \cap (G - NG') \), for some nonlinear character \( \eta \in \Irr(G) \setminus \{\theta_1, \theta_2, \chi\} \) and a non-linear irreducible character \( \gamma \neq \eta \), we arrive at a contradiction since \( \eta \) does not connect with any character. Therefore, when \( g \in (G - NG') \cap \Nv(\eta) \), it must follow that other non-linear characters vanish at \( g \). According to the second orthogonality relation, we can express this as:
   % 	\[
 %  	\sum_{\gamma \in \Irr(G)} \gamma(1) \gamma(h) = \eta(1) \eta(g) + \sum_{\gamma \in \Irr(G/G')} \gamma(g) = 0,
  % 	\]
    %	which leads us to conclude that \( \eta(g) = 0 \), resulting in a contradiction.
   	Hence, by Step 1,  we establish that \( \Van(\eta) = \Van(G) \) for all non-linear characters \( \eta \in \Irr(G) -( \{\theta_1, \theta_2\}\cup \Irr(G/N|G'N/N)) \), as required. 
   	
   	\smallskip
   	{\bf Step 3.}  $\theta_2=\zeta\theta_1$ for some $1\not= \zeta\in \Irr(G/G')$.
   	
   	\smallskip
   	By Step 1, there exists $g\in G-G'N$ such that $\theta_1(g)\not=0. $ As $|G/G'|\geq 2$, then $\zeta\theta_1(g)\not=0$ for all $\zeta\in \Irr(G/G')$, and by Lemma \ref{newchar}, for some $\zeta\in \Irr(G/G')$ we get that $\zeta\theta_1\not=\theta_1$,  implying that $\theta_2=\zeta\theta_1$, as wanted. 
   	
   	\smallskip
   	
   	{\bf Step 4.}   $G'N=G'$.
   	
   	\smallskip
   	
   Assume that \( G'N \neq G' \). For every \( \lambda \in \text{Irr}(G'N/G') \) and every \( \chi \in \text{Irr}(G/N|G'N/N) \), we find that \( \lambda\chi \neq \chi \) because \( \chi \) does not vanish on any element in \( G'N \). Since \( \lambda\chi \) is not in \( \text{Irr}(G/N) \) and \( \text{Van}(\lambda\chi) = \text{Van}(\chi) = G - G'N \), we can conclude that for all \( \chi \in \text{Irr}(G|G') - \{\theta_1, \theta_2\} \), \( \text{Van}(\chi) = \text{Van}(G) \), by Steps 1 and 2. According to Step 3, this matches the condition of Case 1 of Theorem B, which leads to a contradiction.
   	
   	\smallskip
   	
   		{\bf Step 5.} $\Irr(G/N)$ has exactly one non-linear character.
   		
   		\smallskip  
   	
   	Assume instead that \(|\Irr(G/N|G'/N)|\geq 2\). According to \cite[Lemma 12.3]{Isaacs}, we have \(\Van(\chi)=G-G'N=G-G'\) for every non-linear character \(\chi\in \Irr(G/N)\). This means that all these characters must be adjacent, or it could be that \(\Van(G)=\Van(\chi)\) for all \(\chi\in \Irr(G/N|G'/N)\). Since the first scenario is not possible we find that for all \(\chi\in \Irr(G|G')-\{\theta_1, \theta_2\}\), we have  \(\Van(\chi)=\Van(G)\). This creates a contradiction, as we assumed that Case 1 of Theorem B does not hold.
   	
   	\smallskip
   	
   	$\bullet$ In the sequel by $\chi$, we refer to the unique character in $\Irr(G/N|G'/N)$. Moreover in the following Steps we   assume that $M$ is a normal subgroup of $G$ that is contained in $N$, where $N/M$ is a chief factor of $G$. Note that $\Irr(G/M|G'/M)$ has more than one character. Thus by Steps 1 and 2, we get that $\Van(G)\cap M=\emptyset$, as otherwise  $\Irr(G/M|G'/M)$ forms a complete subgraph. 
   	
   	\smallskip

   	{\bf Step 6.} $G/N$ is a Frobenius group of order $p^a(p^a-1)$ with cyclic Frobenius  complement  of order $p^a-1$ and elementary abelian Frobenius kernel of order $p^a$. 
   	\smallskip
   	
  Note that by Step 4, we have $N \subset G' \subset G$. Moving to Step 5, it follows that $G/N$ is either an extraspecial 2-group of order $2^{2n+1}$, for some integer $n$, or a Frobenius group of order $(p^a-1)p^a$, where $p$ is a prime and $a$ is an integer, with abelian Frobenius kernel of order $p^a$ and cyclic Frobenius complement. We will proceed under the assumption that the first case holds.
  
  First assume  $N/M$ is a 2-group. Then $N/M$ must be contained in ${\bf Z}(G/M)$. Should it happen that $N/M \neq {\bf Z}(G/M)$, we find that $G'/M$ is also contained in ${\bf Z}(G/M)$. Consequently, for all characters $\eta \in \Irr(G/M|G'/M)$, we have $\Van(\eta) \subseteq G - G'$. Thus, by Steps 1 and 2, we conclude $\Van(G) = G - G' = \Van(\eta)$ for all $\eta \in \Irr(G|G') - \{\theta_1, \theta_2\}$, leading to a contradiction as it satisfies the conditions of Case 1 of Theorem B. Therefore, we can assume that $N/M = {\bf Z}(G/M)$ has order 2, which means that $G - N \subseteq \Van(G)$. If $\eta \in \Irr(G/M|G'/M) - \{\chi, \theta_1, \theta_2\}$, then $\Van(\eta) = G - N$ based on Steps 1 and 2.  From Lemma \ref{pgroup}, we get a contradiction. 
  %%%%%%%%%%%%%%%%%%%
 % From \cite[problem 6.3]{Isaacs}, we find $\eta(1) = \sqrt{2^{2n+1}}$, which gives us a contradiction. 
 Thus, we can assume that $\Irr(G/M|G'/M) = \{\chi, \theta_1, \theta_2\}$.  Noting that $\theta_2 = \zeta \theta_1$ for some $\zeta \in \Irr(G/G')$, by following \cite[Theorem 7 and Lemma 2]{Ber96}, the only possibility left is that $2^{2n+2} = |G/M| = 2^{2c+3}$ for some integer $c$, which is not feasible. So we may assume there is no chief factor  $N/M$ of $G$ which is a $2$-group. 
  
  Next, let's consider the case where $N/M$ is not a 2-group. First assume  $C_{G'/M}(N/M) \not\subseteq N/M$. In this scenario, $G'/M$ must be abelian of order $2|N/M|=2q^c$, for some prime $q$ and integer $c$. Applying Lemma \ref{pq}, we conclude that   $G'/M \cap \Van(G/M) = \emptyset$. For all characters $\eta \in \Irr(G/M|G'/M)$, we again have $\Van(\eta) \subseteq G - G'$. Thus, by Steps 1 and 2, it follows that $\Van(G) = G - G' = \Van(\eta)$ for all $\eta \in \Irr(G|G') - \{\theta_1, \theta_2\}$, which is a contradiction. Hence, we conclude that $C_{G'/M}(N/M) = N/M$. According to \cite[Lemma 2.9]{silvio1}, we find that $G' - N \subseteq \Van(G)$. By Steps 1 and 2 and utilizing \cite[Corollary 2.2]{BDS}, we deduce that all non-linear irreducible character degrees are divided by $2$ and so  Thompson's theorem leads to that $G$ is $2$-nilpotent.   
    Since there is not any chief factor $N/M_0$ that is a 2-group based on the earlier discussion, we conclude that $P \cong G/N$, where $P\in \Syl_2(G)$. Moreover, as $G'-N\subseteq \Van(G)$, using Steps 1 and 2 and  applying the second orthogonality relation on each  element $x\in G'-N$,  we get that $|C_G(x)|=|G'/N|$ and so  we deduce that $G'/N\cong C_2$ is acting Frobeniusly on $N$ and thus  $N$ is abelian. Thus $N-M\subseteq \Nv(G)$  First assume $\Irr(G/M)\cap \{\theta_1, \theta_2\}=\emptyset$. Then by Steps 1 and 2, we get that $\Van(G/M)=\Van(\eta)$ for all $\eta\in \Irr(G/M|G'/M)-\{\chi\}$. 
     Now applying Lemma \ref{extraspecial} we get that $G/M$ is a Frobenius group with Frobenius complement $G/N\cong Q_8$. In this case we have that for each $y\in P-{\bf Z}(P)$, $y^2=x\in {\bf Z}(P)\subseteq G'-N\subseteq \Van(G)$. Thus $C_G(y)\subseteq C_G(x)\subseteq P$.  Thus $P\cong Q_8$ is acting Frobeniusly on $N$     
      and so  $G-G'\subseteq \Van(\chi)$ for all $\chi\in \Irr(G|G')$ contradicting Step 1.

      Thus we may assume $\{\theta_1, \theta_2\}\subseteq \Irr(G/M)$. From now on we may assume $M=1$ and $N$ is a  minimal normal subgroup of $G$ of order $q^c$. If $P$ acts Frobeniusly on $N$ by the earlier discussion  we get a contradiction.           
      Hence, \( P\cong G/N \) cannot act Frobeniusly on \( N \), while the action of $G'/N$ on $N$ is in a Frobenius manner. Thus, there exists some element \( x \in N \) such that \( |C_{G}(x)/N| \geq 2 \). 
             This indicates that there is a character \( \lambda \in \Irr(N) \) such that \( |I_{G}(\lambda)/N| \geq 2 \) and $I_G(\lambda)\cap G'=N$. Note that  \( \lambda^{G'} \in \Irr(G') \) is the only character above \( \lambda \) in \( G' \). Note that by  \cite[Theorem 6.26]{Isaacs}, $\lambda$ extends to some characters $\lambda_0\in I_G(\lambda)$. Note that \( I_{G}(\lambda^{G'}) /G'= I_{G}(\lambda)G' /G' \) is a  group of order at least $2$. Thus by \cite[Problem 5.2]{Isaacs}, $(\lambda_0^{I_G(\lambda^{G'})})_{G'}=({\lambda_0}_{N})^{G'}=\lambda^{G'}$. On the other hand by Clifford's theorem $\lambda_0^{I_G(\lambda^{G'})}\in \Irr(I_G(\lambda^{G'}))$.  Thus we have that $\lambda^{G'}$ extends to $I_{G}(\lambda^{G'})$.

      We claim that for any irreducible character of $G$ above $\lambda^{G'}$, say $\gamma$,  there is at least one element  \( g\in I_{G}(\lambda^{G'})-G' \) such that $\gamma(g)\not =0$. Assume otherwise. Note that $\gamma=\eta^{G}$, where $\eta$ is an extension of $\lambda^{G'}$ to $I_{G}(\lambda^{G'})$, by Clifford's correspondence \cite[Theorem 6.11]{Isaacs}.  Thus by our assumption $\gamma$ vanishes  on $G-G'$.  	Assume $\{g_1, \dots g_{t}\}$ is a set of right transversal of $I_{G}(\lambda^{G'})$ in $G$ and $t=|G:I_{G}(\lambda^{G'})|$. Therefore applying \cite[Lemma 2.29]{Isaacs}, we have  $$|G/G'|=|G/G'|[\gamma,\gamma]=[\gamma_{G'}, \gamma_{G'}]=[\sum_{i=1}^{t}( \lambda^{G'})^{g_i}, \sum_{i=1}^{t}( \lambda^{G'})^{g_i}]=t,$$
      which is a contradiction. Thus our claim is valid. 
      
      Note that according to Gallagher's  Theorem and Clifford's correspondence (see \cite[Theorem 6.11 and 6.17]{Isaacs})  there are $|I_{G}(\lambda^{G'})/G'|$ distinct irreducible characters above $\lambda^{G'}$ in $G$. Therefore, all those  irreducible characters above \( \lambda^{G'} \) in \( G\) must be either \( \theta_1 \) or \( \theta_2 \), as they are the only non-linear irreducible characters that do not vanish on some elements in \( G - G' \), according to Steps 1 and 2. Hence $|I_{G}(\lambda)/N|=2$. 
      
      Based on this discussion, all characters \( \lambda \in \Irr(N) \) with the property \( |I_{G}(\lambda)/N| =2\) must be \( G \)-conjugate. Otherwise, there would be other characters besides \( \theta_1 \) and \( \theta_2 \) that do not vanish on some elements in \( G - G' \). Therefore, the action of \( G/N\) on \( N \cong \Irr(N) \) has exactly one orbit of size \(|P|/2\) and other orbits have size $|P|$. 
      
     Recall that \( P \) is an extraspecial 2-group with  order  \( 2^{2n+1} \). Using Lemma \ref{orbit}, we find that \( 2(q^{c_0} - 1) = |P|/2 = 2^{2n} \) for some integer \( c_0 \). According to Zsigmondy's theorem, either \( c_0 = 1 \) or \( q^{c_0} = 3^2 = |P|/4 + 1 = 2^{2n-1} + 1 \). 
     
     In the first case, \( q - 1 = 2^{2n-1} \); this scenario is not possible unless \( n = 1 \) and \( q = 3 \).
          For the second case, we conclude that \( n = 2 \). In both instances, we find that \( q = 3 \). Thus, the orbit sizes for the action of \( G/N \) on \( N \cong \Irr(N) \) are either \( 2^{2n+1} \) or \( 2^{2n} \), when \( n = 1 \) or \( n = 2 \). This leads to the equation \( q^c - 1 = 3^c - 1 = 2^{2n} + k(2^{2n+1}) = 2^{2n} \times (1 + 2k) \) for some integer \( k \).
     
     If \( n = 1 \), then \( 4 \) divides \( 3^c - 1 \) while \( 8 \) does not divide \( 3^c - 1 \). This gives that \(c \) is even, meaning \( 3^c - 1 \) is divisible by \( 8 \), which leads to a contradiction. 
     
     If \( n = 2 \), we have \( |G| = 32 \times 3^c \). However, the only irreducible  $\mathbb{F}_3P$-modules  have dimensions  \( 1 \) or \( 4 \). Since \( 3^c - 1 = 16(2k + 1) \), \( c \) must be divisible by \( 4 \). Therefore, the only possible value for \( |G| \) is \( 3^4 \times 2^5 \). Upon checking with GAP, we find that no such groups exist with those properties. 
    
    \smallskip
    
    % Let \( x \in \text{Z}(P) \). This indicates that all non-linear irreducible characters, with the exception of \( \chi \). By applying the second orthogonality relation, we can conclude that \( |C_G(x)| = |G:G'| + |\chi(x)|^2 = |G/N| = |P| \). Now, consider \( y \in P - \text{Z}(P) \). In this case, we find that \( y^2 = x \). Therefore, we have \( C_G(y) \subseteq C_G(x) \subseteq P \), which shows that \( P \) is acting Frobeniusly on \( N \). It follows that  \( P \) is a quaternion group. This implies that \( P \cong Q_8 \). Consequently, \( \Gamma_{nv}(G) \) is null, which presents a contradiction.
    
   	{\bf Step 7. }  Let $L$ be any normal subgroup of $G$ contained in $G'$. Then  $|C_{G/L}(hL)|=|G/G'|$ for all elements $h\in G-G'$ such that $hG'$ has order greater than 2. In particular Hall $2'$-subgroup of $G/G'$ is acting on $G'/L$,  Frobeniusly.  
   	
   \smallskip
   	
   	Recall that by Step 6, \( G/G' \) is a cyclic group of order \( p^a - 1 \), where \( p \) is a prime and \( a \) is an integer. Assume \( h \in G - G' \) such that the order of \( hG' \) in \( G/G' \) is greater than 2. If \( \theta_1(h) \neq 0 \), then by Lemma \ref{newchar}, there exist at least three non-linear irreducible characters that do not vanish at \( h \), leading to the conclusion that a triangle must exist, creating a contradiction. Therefore, we must assume \( \theta_1(h)  = 0 \). Hence by Steps 2-4,  all non-linear irreducible characters of $G/L$ are vanishing on $hL$ and applying the second orthogonality relation on $hL$ we get that $|C_{G/L}(hL)|=|G/G'|$ as desired.
   	
   	Let \( H/G' \) be the Hall \( 2' \)-subgroup of \( G/G' \). Consequently, \( |C_{H/L}(hL)| = |H/G'| \) for every \( h \in H - G' \). By applying Lemma \ref{Frobenius}, we conclude that \( H/G' \) acts Frobeniusly on \( G'/L \), as required.

   		% As $hG'$ has order $|G/G'|$ we get that $h$ has order $|G/G'|$  and $C_G(h)=\langle h\rangle$.
   	%	This implies that $C_G(h)=\langle h\rangle $ for all elements $h\in G-G'$ such that $hG'$ has order greater than 2. 
   	
   	\smallskip
   		 
   		{\bf Step 8.} $N/M$ is not a $p$-group. 
   		
   		\smallskip

   	Assume instead that \( |N/M| = p^b \), for some integer $b$, indicating that it is a \( p \)-group. We know that \( \Irr(G/M|G'/M) \) contains at least two characters. Moreover, it holds that \( \Van(G) \cap G' \subseteq \Van(\theta) \) for every \( \theta \in \Irr(G|G') - \{\chi\} \). If \( \Van(G) \cap (G' - M) = \emptyset \), then it follows that \( \Van(G) = G - G' \), by Steps 1 and 2. Consequently, the vanishing set of all non-linear irreducible characters, except for \( \theta_1 \) and \( \theta_2 \), would coincide with \( \Van(G) \), leading us to Case 1 of Theorem B. Therefore, we may assume that \( \Van(G) \cap (G' - M) \neq \emptyset \).  From this point, we will focus on \( \tilde{G} = G/M \).   By applying Lemma \ref{pp}, we have that $b<a$. 
   	
   %%%%%%%%%%%%%%%%%%%%%%%%%%%%%%%%%   	
   %	From this point, we will focus on \( \tilde{G} = G/M \). Thus, we have \( \tilde{N} \subseteq \mathbf{Z}(\tilde{G'}) \).   	
  % 	Now, let \( g \in \tilde{N} \cap \Van(\tilde{G}) \). We have the following equation:
   	
  % 	$$ p^{b+a} = |C_{\tilde{G}}(g)| = |G:G'| + |\chi(g)|^2 = |G:G'| + \chi(1)^2 = p^a - 1 + (p^a - 1)^2,$$
   	
 %  	which presents a contradiction. Hence, we can assume \( g \in \tilde{G'} - \tilde{N} \). By applying the second orthogonality relation to \( g \), we derive:
 %  	
 %  	\[
 %  	p^b = |\tilde{N}| < |C_{\tilde{G}}(g)| = \sum_{\theta \in \Irr(\tilde{G})} |\theta(g)|^2 = |\chi(g)|^2 + |G:G'| = (-1)^2 + p^a - 1 = p^a,
 %  	\]
  % 	   	where the first inequality follows from \( \tilde{N} \subseteq \mathbf{Z}(\tilde{G'}) \). Therefore, we deduce that \( b < a \).
%%%%%%%%%%%%%%%%%%%%%%%%%%%%%%%%%%%%   %	
   	First, we assume \( p^a - 1 = 2^{\beta} \) for some integer \( \beta \). This leads us to consider two scenarios, by Zsigmondy's  Theorem: either \( p^a = 3^2 \) or \( a = 1 \). Given that \( b < a  \), the second case cannot occur. Therefore, we conclude that \( a = 2 \), \( b = 1 \), and \( p = 3 \). Consequently, \( |\tilde{G}| = 216 \), and upon checking with GAP \cite{gap}, we find that \( \Gamma_{nv}(\tilde{G}) \) contains more than one edge, which is a contradiction.
   	
   	Thus $G/G'$ is not a $2$-group. Following Step 7, we note that the Hall \( 2' \)-subgroup of \(\tilde{G}/\tilde{G'}\cong G/G' \) acts Frobeniusly on \( \tilde{G'} \) and so on $\tilde{N}$. This indicates that \( (p^a - 1)_{2'} \mid (p^b - 1) \). Since \( b < a \), it follows that \( p^a - 1 \) does not possess any primitive prime divisor. Thus, Zsigmondy's theorem (see \cite[Theorem V]{zsig}) implies one of the following: either \( a = 1 \); \( a = 2 \) with \( p + 1 \) being a power of \( 2 \); or \( p = 2 \) and \( a = 6 \). The first case is impossible, as it would imply \( 1 < b < a = 1 \), and the last case does not hold since, in that scenario, \( (p^a - 1)_{2'} = (p^a - 1) \mid (p^b - 1) \) with \( b < a \). Thus, the only viable option is \( a = 2 \) and \( b = 1 \). 
   %%%%%%%%%%%%%%%%%%%%%%%%%%%%%%%	
   		It’s important to note that if \( \tilde{G'}\) is abelian, then none of the irreducible character degrees would be divisible by \( p \). By referencing \cite[Corollary 2.2]{BDS}, we find that \( \Van(G) \cap (G' - M) = \emptyset \), which is a contradiction. 
   	Thus  \( \tilde{G'} \) is a non-abelian group of order \(|\tilde{G'}|=|\tilde{G'}/\tilde{N}||\tilde{N}|=p^2\times p= p^3 \).  	
   	Therefore, \( \tilde{G'} \) has exactly \( p - 1 \) non-linear irreducible characters, and all of them have degree $p$.
   	%%%%%%%%%%%%%%%%%%%%%%%%%%
   	% all  of them of  degree \( p \), and all these characters vanish on \( \tilde{G'} - \tilde{N} \).
   	
   		Referring back to Step 7, for each non-trivial \( x \in \tilde{N} \), we have \( \tilde{G'} \subseteq C_{\tilde{G}}(x) \subseteq \tilde{K} \), where \( \tilde{K} \) is a subgroup of \( \tilde{G} \) with order \( 2|\tilde{G'}| \). Consequently, for every non-trivial character \( \lambda_0 \in \Irr(\tilde{N}) \), we derive \( \tilde{G'} \subseteq I_{\tilde{G}}(\lambda_0) \subseteq \tilde{K} \). Given that \( |\tilde{N}| = p \), it follows that for each \( \lambda_0 \in \Irr(\tilde{N}) \), \( p^3 \) divides \( |I_{\tilde{G}}(\lambda_0)| \), and \( |I_{\tilde{G}}(\lambda_0)| \) divides \( 2p^3 \).
   	 Let \( \lambda \in \Irr(\tilde{G'}) \) be a non-linear character  and assume 
   	 \( \lambda \) is above \( \lambda_0 \in \Irr(\tilde{N}) \). 
   	 %then \( I_{\tilde{G}}(\lambda) = I_{\tilde{G}}(\lambda_0) \).
   	   	Since \( \lambda_{\tilde{N}} = p \lambda_0 \), we have \( \tilde{G'}\subseteq  I_{\tilde{G}}(\lambda) \subseteq I_{\tilde{G}}(\lambda_0) \). %Let \( h \in I_{\tilde{G}}(\lambda_0) \). Now, take \( x \in \tilde{G'} - \tilde{N} \), then \( \lambda^{h}(x) = \lambda(x^{h^{-1}}) = 0 = \lambda(x) \). When \( x \in \tilde{N} \), it follows that \( \lambda^h(x) = p \lambda_0^h(x) = p \lambda_0(x) = \lambda(x) \), which implies \( h \in I_{\tilde{G}}(\lambda) \), as claimed.   	
    This leads to \(  |\tilde{G}:I_{\tilde{G}}(\lambda)| \in \{ p^2 - 1, (p^2 - 1)/2 \} \). On the other hand, since the number of non-linear irreducible characters of \( \tilde{G'} \) equals \( p - 1 \), it follows that \( |\tilde{G}:I_{\tilde{G}}(\lambda)| \leq p - 1 \), which results in a contradiction.
   	    	 
   	 % Since $|\tilde{N}|=p$ we get that $I_{\tilde{G}}(\lambda)$ for all non-linear character
   	    %	 	 $\lambda\in\Irr(\tilde{G'})$ are the same.  According to Step 7, for each $\lambda_0\in \Irr(\tilde{N})$   Thus $|\tilde{G}:I_{\tilde{G}}(\lambda_0)|=|\tilde{G}:I_{\tilde{G}}(\lambda)|\in \{p^2-1, (p^2-1)/2\}$.
   	      \smallskip
   	      
   		  		{\bf Step 9.}  $G/M\cong S_4$. In particular $M=\ker\theta_1$.
   		
   		\smallskip

   	By Step 8,  $\tilde{G'}\cong G'/M$ is not a $p$-group. Thus, we can assume that $|\tilde{N}|=q^{c}$ for some integer $c$ and prime $q\not =p$.
   	
   	 First let $\tilde{G'}$ be  abelian. Then $\tilde{G'}$ can be expressed as $\tilde{N}\times \tilde{K}$, where $\tilde{K}$ is a normal subgroup of $\tilde{G'}$ isomorphic to $\tilde{G'}/\tilde{N}=G'/N$ with order $p^a$. 
   	 
   	 First assume $q\mid |G/G'|$. Then $\tilde{G}/\tilde{K}$ has at least one non-linear irreducible  character. Moreover as $G/G'\cong \tilde{G}/\tilde{G'}$ is cyclic, then the Sylow $q$-subgroup of $\tilde{G}/\tilde{K}$ is normal and abelian. Thus no irreducible character degree of $\tilde{G}/\tilde{K}$ is divisible by $q$ and so no irreducible character of $\tilde{G}/\tilde{K}$ is vanishing on $q$-elements of $\tilde{G}/\tilde{K}$. Hence no irreducible character of $\tilde{G}/\tilde{K}$ is vanishing on  whole  $\tilde{G}-\tilde{G'}$. So by Steps 1-4,  $\tilde{G}/\tilde{K}$ has exactly two non-linear irreducible characters $\theta_1$ and $\theta_2$.     As  we are in the situation that $\ker\theta_1=\ker\theta_2$,  using Theorem \ref{key3} we get a contradiction. 
   	 
   	Therefore,  we may assume $(q, |G/G'|)=1$.  We find that the Sylow $q$-subgroup of $\tilde{G}$ is  normal and abelian  This leads us to the conclusion that $\tilde{G'}/\tilde{K} \subseteq \Nv(\tilde{G}/\tilde{K})$, as per \cite[Corollary 2.2]{BDS}. Hence $\Van(\theta)\subseteq G-G'$, for all non-linear character $\theta\in \Irr(\tilde{G}/\tilde{K})$.
   	%%%%%%%%%%%%%%%%%%%%%%%%
   %	If $q\mid |G/G'|$, then the Hall $p'$-subgroup of $\tilde{G}$ has a normal and abelian Sylow $q$-subgroup, as $\tilde{G}/\tilde{G'}$ is abelian.  Therefore no irreducible character degree of the Hall $p'$-subgroup of $\tilde{G}$ is divided by $q$. This implies that Sylow $q$-subgroup of Hall $p'$-subgroup of $\tilde{G}$ must be contained in $\Nv(\tilde{G})$. 
    Consequently, $(G'-M) \cap \Van(G) = \emptyset$, implying that $\Van(G) = \Van(\theta) = G - G'$ for all $\theta \in \Irr(G|G')-\{\theta_1, \theta_2\}$, by Steps 1 and 2, which contradicts our assumption.
    
     Thus we assume $\tilde{G'}$ is not abelian, implying that $C_{\tilde{G}}(\tilde{N})=\tilde{N}$.     
     Therefore, we must have $\tilde{G'}-\tilde{N} \subset \Van(G)$, according to \cite[Lemma 2.9]{silvio1}. Thus by Step 1, for every $g \in \tilde{G'} - \tilde{N}$, it follows that $|C_{\tilde{G'}}(g)| = \sum_{\theta\in \Irr(G/N)}|\theta(g)|^2 =|C_{G/N}(gN)|=p^a$. Thus, $\tilde{G'}/\tilde{N}$ acts on $\tilde{N}$  Frobeniusly. As a result, $\tilde{G'}/\tilde{N}$ must be cyclic, leading us to conclude that $a = 1$. Thus  $\tilde{G}/\tilde{N}$ is a Frobenius group of order $p(p-1)$.

   	Let \( p-1 \) have an odd divisor. Then the Hall \( 2' \)-subgroup of \( G/G' \cong \tilde{G}/\tilde{G'} \) acts on \( \tilde{G'} \) Frobeniusly, as established in Step 7. This implies that \( \tilde{G'} \) is nilpotent, which leads to 
   	   	 a contradiction, as $\tilde{G'}$ is a Frobenius group.
   	   	Consequently, we conclude that \( p-1 = 2^{\alpha} \) for some integer \( \alpha \), implying that $|\tilde{G}|=2^{\alpha}(2^{\alpha}+1)q^c$. It is widely known that a Frobenius group cannot act Frobeniusly on any group. Hence, \( \tilde{G}/\tilde{N} \) cannot act Frobeniusly on \( \tilde{N} \), while the action of $\tilde{G'}/\tilde{N}$ on $\tilde{N}$ is in a Frobenius manner. Thus, according to Step 7, there exists some element \( x \in \tilde{N} \) such that \( |C_{\tilde{G}}(x)/\tilde{N}| = 2 \). This indicates that there is a character \( \lambda \in \Irr(\tilde{N}) \) such that \( |I_{\tilde{G}}(\lambda)/\tilde{N}| = 2 \). Note that  \( \lambda^{\tilde{G'}} \in \Irr(\tilde{G'}) \) is the only character above \( \lambda \) in \( \tilde{G'} \). As \( I_{\tilde{G}}(\lambda^{\tilde{G'}}) /\tilde{G'}= I_{\tilde{G}}(\lambda)\tilde{G'} /\tilde{G'} \) is a cyclic group of order $2$, then $\lambda^{\tilde{G'}}$ extends to $I_{\tilde{G}}(\lambda^{\tilde{G'}})$, by \cite[Corollary 11.22]{Isaacs}.
   	   	 
   	   	We claim that for any irreducible character of $\tilde{G}$ above $\lambda^{\tilde{G'}}$, say $\gamma$,  there is at least one element  \( g\in I_{\tilde{G}}(\lambda^{\tilde{G'}})-\tilde{G'} \) such that $\gamma(g)\not =0$. Assume otherwise. Note that $\gamma=\eta^{\tilde{G}}$, where $\eta$ is an extension of $\lambda^{\tilde{G'}}$ to $I_{\tilde{G}}(\lambda^{\tilde{G'}})$, by Clifford's correspondence \cite[Theorem 6.11]{Isaacs}.  Thus by our assumption $\gamma$ vanishes  on $\tilde{G}-\tilde{G'}$.  	Assume $\{g_1, \dots g_{|\tilde{G}/\tilde{G'}|/2}\}$ is a set of right transversal of $I_{\tilde{G}}(\lambda^{\tilde{G'}})$ in $\tilde{G}$. Therefore applying \cite[Lemma 2.29]{Isaacs}, we have  $$|\tilde{G}/\tilde{G'}|=|\tilde{G}/\tilde{G'}|[\gamma,\gamma]=[\gamma_{\tilde{G'}}, \gamma_{\tilde{G'}}]=[\sum_{i=1}^{|\tilde{G}/\tilde{G'}|/2}( \lambda^{\tilde{G'}})^{g_i}, \sum_{i=1}^{|\tilde{G}/\tilde{G'}|/2}( \lambda^{\tilde{G'}})^{g_i}]=|\tilde{G}/\tilde{G'}|/2,$$
     which is a contradiction. Thus our claim is valid. 
   	
   	  Note that according to Gallagher's  Theorem and Clifford's correspondence (see \cite[Theorem 6.11 and 6.17]{Isaacs})  there is only two characters above $\lambda^{\tilde{G'}}$in $\tilde{G}$. Therefore, the two irreducible characters above \( \lambda^{\tilde{G'}} \) in \( \tilde{G} \) must be \( \theta_1 \) and \( \theta_2 \), as they are the only non-linear irreducible characters that do not vanish on some elements in \( G - G' \), according to Steps 1 and 2. 
   	
   	Based on this discussion, all characters \( \lambda \in \Irr(\tilde{N}) \) with the property \( |I_{\tilde{G}}(\lambda)/\tilde{N}| = 2 \) must be \( \tilde{G} \)-conjugate. Otherwise, there would be other characters besides \( \theta_1 \) and \( \theta_2 \) that do not vanish on some elements in \( G - G' \). Therefore, the action of \( \tilde{G}/\tilde{N} \) on \( \tilde{N} \cong \Irr(\tilde{N}) \) has exactly one orbit of size \((p-1)p/2= 2^{\alpha - 1}p \) and other orbits have size $(p-1)p=2^{\alpha}p$, by Step 7. 
   	
   	By applying Lemma \ref{orbit}, we have \( p(q^{c_0} - 1) = 2^{\alpha - 1}p \) for some integer \( c_0 \). Thus, Zsigmondy's theorem  implies that either \( c_0 = 1 \) or \( q^{c_0} = 3^2=2^{\alpha-1}+1=(p+1)/2 \), which leads to \( p = 17 \). In the second case, we find that the orbit sizes for the action of \( \tilde{G}/\tilde{N} \) on \( \tilde{N}\cong \Irr(\tilde{N}) \) are either \( 2^{\alpha}p = 8 \times 17 \) or \( 2^{\alpha - 1}p = 16 \times 17 \). Therefore, \( q^c - 1 = 3^c - 1 = 17 \times 8 + k(17 \times 16) = 17 \times 8(1 + 2k) \) for some integer \( k \).     	
   	As $17\mid (3^c-1)$, we have  \( 16 \mid c \), and consequently \( 3^c - 1 \) is divisible by \( 16 \), which leads to a contradiction. Hence, we conclude that  $c_0=1$ and \( q = 2^{\alpha - 1} + 1 \) while \( p = 2^{\alpha} + 1 \), which is only possible if \( \alpha = 1, 2 \), leading to the conclusion that  \( (p,q) \in \{(3,2), (5,3)\}  \). For the second case, we have the equation \(3^c - 1 = \frac{p(p-1)}{2} + k \times (p(p-1)) = 10 + 20k = 10(2k + 1)\), which means that \(c\) must be odd. However, since \(\tilde{G'}\) is a Frobenius group with a Frobenius kernel \(\tilde{N}\), it follows that \(5\) divides \(3^c - 1\). This indicates that \(c\) must be even, which creates a contradiction. Hence $q=2$ and $p=3$. 
   	
  Thus, we have that \(\tilde{G}/\tilde{N} \cong S_3\), which means \(\tilde{N}\) is an irreducible \(\mathbb{F}_2 S_3\)-module. The possible dimensions for \(\tilde{N}\) are 1 or 2, leading us to conclude that \(c = 1\) or \(c = 2\). Given that \(\tilde{G'}/\tilde{N} \cong C_3\) acts on \(\tilde{N}\) Frobeniusly, we find that \(c = 2\) and thus \(\tilde{G} \cong S_4\), as we aimed to show. Since none of the non-linear irreducible characters of \(S_4\) vanish on \(\text{Van}(S_4)\), we can deduce that \(\{\theta_1, \theta_2\} \subseteq \Irr(\tilde{G})\), which implies that \(M = \ker \theta_1\).

   \smallskip
   	 
   	 {\bf Step 10.} $\ker\theta_1=1$. In particular $G\cong S_4$.
   	 
   	 \smallskip

   	Assuming, for the sake of contradiction, that \(\ker\theta_1 \neq 1\). By Step 9, we have \(G/\ker\theta_1 \cong S_4\). Without loss of generality, we can assume \(\ker\theta_1\) to be a minimal normal subgroup of \(G\). Since both \(\theta_1\) and \(\theta_2\) vanish on \(G' - N\), we can refer to Steps 1 and 2, which show that every irreducible character in \(\Irr(G|G') - \{\chi\}\) also vanishes on \(G' - N\) (Recall that $\chi$ is the only non-linear irreducible  character of $G/N$). Therefore, applying the second orthogonality relation, we find that \( |C_G(x)| = |G:G'| + |\chi(x)|^2 = 2 + (-1)^2 = 3 = |G'/N| \) for all \(x \in G' - N\). This indicates that \(G'/N\) acts on \(N\)  Frobeniusly, by Lemma \ref{Frobenius}, leading us to conclude that \(N\) is nilpotent.

   	 First  we  assume that \(\ker\theta_1\) is a \(2\)-group, leading us to the conclusion that \(\ker\theta_1 \subseteq \mathbf{Z}(N)\).
   	%%%%%%%%%%5
   	% From the minimal nature of \(\ker\theta_1\), it follows that \(N' = 1\) or \(N' = \ker\theta_1\).
   	%%%%%%%%%%%%%%%%5
   	%We show that the Sylow \(3\)-subgroups of \(G\)   act Frobeniously on \(\ker\theta_1\). Suppose there exists an element \(x\in \ker\theta_1\) for which the order of \(C_G(x)\) is divisible by \(3\). This implies that \(x\) must belong to \({\bf Z}(L)\), where \(L/N\) is a chief factor of \(G\) isomorphic to \(C_3\). Consequently, \(\ker\theta_1 \cap {\bf Z}(L) \neq 1\), which suggests that \(\ker\theta_1 \subseteq {\bf Z}(L)\). Additionally, since \(\ker\theta_1 \cap {\bf Z}(Q) \neq 1\) for any \(Q \in \Syl_2(G)\), it follows that \(\ker\theta_1 \cap {\bf Z}(LQ) \neq 1\). Given that \(LQ = G\), we deduce \(\ker\theta_1 \subseteq {\bf Z}(G)\), concluding that \(|\ker\theta_1| = 2\). Therefore, \(|G| = 48\), and upon checking with GAP, we reach a contradiction.
   	   	 Let $1\not =x\in \ker\theta_1\cap {\bf Z}(Q)$ where $Q\in \Syl_2(G)$. Thus  $|x^G|=3$, as $G'/N$ acts on $N$ Frobeniusly. As $\ker\theta_1=\langle x^G\rangle $, we get that $|\ker\theta_1|=2^a$ for some $a\leq 3$. Recalling that  $G'/N\cong C_3$ is acting Frobeniusly on $\ker\theta_1$, we have $3\mid(2^a-1)$  and so $a=2$. Thus $|G|=96$ and  upon checking with GAP, we reach a contradiction.

   	Now, let's examine the situation where \((|\ker\theta_1|, |N/\ker\theta_1|) = 1\). Given that \(N\) is nilpotent, we have \(C_N(\ker\theta_1) = N\). Therefore, we can express \(N\) as \(N = \ker\theta_1 \times K\), where \(K \cong N/\ker\theta_1\). This means that \(\Irr(G/K)\) does not include \(\theta_1\) or \(\theta_2\), but it does contain \(\chi\in \Irr(G/K/N/K)\). Considering the structure of \(G/\ker\theta_1 \cong S_4\), we find that \(\Van(\chi) \neq \Van(G/K)\), and there exists an element \(h \in (G' - N) \cap \Van(G)\) such that \(\chi(h) \neq 0\).
   	
   	Recall that by Steps 1 and 2, for all characters \(\eta \in \Irr(G|G') - \{\theta_1, \theta_2, \chi\}\), we have \(\Van(\eta) = \Van(G)\). Note that \(|\Irr(G/K|G'/K)| \geq 2\) and \(\Irr(G/K)\) contains only one character \(\chi\), which satisfies \(\Van(\chi) \neq \Van(G/K)\). According to Theorem A, this situation suggests that \(G/K\) must be isomorphic to a Frobenius group with a \(Q_8\) as a Frobenius complement. However, since \(G/K\) includes a quotient  $G/N\cong S_3$, we arrive at a contradiction.   $\blacksquare$

   	 \smallskip

 \smallskip
  
{\bf The Proof of Theorem D.} By Theorem C, we determine that  $G$ is solvable. Now, we explore the situation where we assume $|\Irr(G|G')|>1$ and $\Gamma_{nv}(G)$ forms a star configuration. Let $M$ be a maximal subgroup among normal subgroup with the property that  $G/M$ is non-abelian. It is important to remind that for all characters $\eta \in \Irr(G/M|G'M/M)$, we have $\Van(\eta)=G-G'M$, by \cite[Lemma 12.3]{Isaacs}. According to Theorem \ref{triple}, we  just need to show that  $|\Irr(G/M|G'M/M)|=1$. 

% and assuming $\Irr(G/M|G'M/M)=\{\chi\}$,  all characters in $\Irr(G|G')-\{\chi\}$ are vanishing on $\Van(G)\cap (G'M-M)$. Moreover there is an element $g_0\in \Van(G)\cap M $ and a unique character $\theta\in \Irr(G|G')-\{\chi\}$ that are not vanishing on $g_0$.  

% First, suppose $|\Irr(G/M|G'M/M)|>2$.  If we find that $\Van(G)\cap MG'\neq \emptyset$, this would lead to $\Irr(G/M|G'M/M)$ forming a complete graph, which presents a contradiction. Thus, we can assume $G'M\subseteq \Nv(G)$. As a result, $\Van(G)=\Van(\eta)$ for all $\eta \in \Irr(G/M|G'M/M)$ indicates the presence of an isolated vertex, which again leads to a contradiction. 

 We are examining the scenario where \( |\Irr(G/M|G'M/M)| \geq 2 \). If there exists a non-linear character \( \chi \in \Irr(G|M) \) such that \( \chi(g) \neq 0 \) for some \( g \in \Van(G) \cap G'M \), this would imply that \( \Irr(G/M|G'M/M) \cup \{\chi\} \) forms a triangle, which would result in a contradiction. Consequently, we deduce that for every non-linear character \( \chi \in \Irr(G|M) \), it must hold that \( \chi \) vanishes on \( \Van(G) \cap G'M \).
  Thus, for all non-linear characters \( \chi \in \Irr(G|M) \) and \( \theta \in \Irr(G/M|G'M/M) \), we establish that \( \Van(\chi) \cup \Van(\theta) = \Van(G) \). This indicates that there is no edges connecting vertices outside of   \( \Irr(G/M|G'M/M) \)  to those within it. Given our hypothesis on $\Gamma_{nv}(G)$, we conclude that \( M = 1 \) and that the vanishing set of all non-linear irreducible characters is \( G - G'M = G - G' \). Consequently, all vertices are isolated, leading to a contradiction.  Thus  \( |\Irr(G/M|G'M/M)| =1 \) as wanted.  $\blacksquare$  


\begin{thebibliography}{1}
	\bibitem{ourself} Z. Akhlaghi, K. Aziziheris,  S. Y. Madanha, Groups having irreducible characters with many zeros, J. Algebra  709 (2027) 749-780.  
\bibitem{Ber96} Y. Berkovich, Finite solvable groups in which only two nonlinear irreducible characters have equal degrees, J. Algebra 184 (1996) 584--603.
\bibitem{BCH92} Y. Berkovich, D. Chillag, M. Herzog, Finite groups in which the degrees of nonlinear irreducible characters are distinct, Proc. Amer. Math. Soc. 115(4) (1992) 955--959.	
\bibitem{ema}
M. Bianchi, D. Chillag, M. L. Lewis, E. Pacifici,
\newblock Character degree graphs that are complete graphs,
\newblock  Proc. Amer. Math. Soc.
135(3) (2007)
671--676.

\bibitem{zsig}
G. D. Birkhoff and H. S. Vandiver,
On the integral divisors of $a^n-b^n$,
Annals of Mathematics
5 (1904) 173--180.

	\bibitem{BDS} D. Bubboloni, S. Dolfi, P. Spiga, Finite groups whose irreducible characters vanish only on $p$-elements, J. Pure and Applied Algebra  213 (2009) 370--376. 

%\bibitem{brough} J. Brough, On vanishing criteria that control finite group, J. Algebra 458 (2016) 207--215.

%\bibitem{Carter} R. Carter, Finite Groups of Lie Type. Conjugacy Classes and Complex Characters(John Wiley and Sons, New York, 1985.

\bibitem{atlas} J.  H. Conway, R. T. Curtis, S. P. Norton, R. A. Parker and R. A. Wilson, Atlas of Finite Groups (Oxford University Press, Eynsham, 1985).



%\bibitem{ema}
%M. Bianchi, D. Chillag, M. L. Lewis and E. Pacifici, ‘Character degree graphs that are complete graphs’, Proc. Amer.
%Math. Soc., 135(3) (2007), 671–676.

%\bibitem{berkovich} Y.G. Berkovich and E.M. Zhmud’, Characters of Finite Groups. Part 1, Translated from the
%Russian manuscript by P. Shumyatsky and V. Zobina, Translations of Mathematical Mono-graphs 172, American Mathematical Society, Providence, RI, 1998.



%\bibitem{BDS} D. Bubboloni, S. Dolfi, P. Spiga, Finite groups whose irreducible characters vanish only on $p$-elements, J. Pure Appl. Algebra, 213(2009) 370--376.

%\bibitem{Atlas} J. H. Conway, R. T. Curtis, S. P. Norton, R. A. Parker, R. A. Wilson, ATLAS of Finite Groups, Oxford Univ. Press (Clarendon), Oxford, New York, 1985.

%\bibitem{selp} S. Dolfi, E. Pacifici, L. Sanus and P. Spiga, On the vanishing prime graph of finite groups, J. London Math. Soc. 82(2) (2016) 167-183.

\bibitem{silvio1} S. Dolfi,
E. Pacifici,
L. Sanus, P. Spiga, On the orders of zeros of irreducible characters, J. Algebra 321 (2009) 345--352.

%\bibitem{DPSS10b} W. Feit, G.M. Seitz, On finite rational groups and related topics, Illinois J. Math. 33 (1989) 103–131.

%\bibitem{silvio} S. Dolfi,  G. Navarro, 
%E. Pacifici,
%L. Sanus,
%P.  H. Tiep,  Non-vanishing elements of finite groups, J. Algebra, 323 (2010), 540--545.

%\bibitem{feit} W. Feit and G. M. Seitz, On finite rational groups and related topics, Illinois Journal of Mathematics 33 (1989), 103--131.

%\bibitem{regularsymmetric} J. B. Fawcett, E. A. \'OBrien, J.  Saxl, Regular orbits of symmetric and alternating
%groups, J. Algebra, 468 (2016), 21--52.

\bibitem{ono} A. Granville and  K. Ono,  Defect zero $p$-blocks for finite simple groups, Trans. Amer. Math. Soc.  438 (1996) 331--348.


%\bibitem{gagula} S. M.  Gagola, Characters vanishing on all but two conjugacy classes Pacific J. Math. 109 (1983), 363–385.

\bibitem{gap}
The GAP~Group, GAP -- Groups, Algorithms, and Programming,
	Version 4.15.1;
2025,
\url{https://www.gap-system.org}.


\bibitem{centraltype} R. B. Howlett and I. M. Isaacs, On groups of central type, Math. Z. 179(4) (1982) 555–569.

\bibitem{hung} N. N. Hung , A. Moreto, L.  Morotti, Common zeros of irreducible characters,  Journal of the Australian Mathematical Society. 117(2) (2024) 105--129.


\bibitem{Isaacs} I. M. Isaacs, {Character Theory of Finite Groups}, New York NY: Academic Press 1976.

\bibitem{JK} G. James and A. Kerber, The Representation Theory of the Symmetric Group, Encyclopedia of
Mathematics and its Applications, vol. 16, Addison-Wesley Publishing Co., Reading, Mass. 1981.




%\bibitem{INW99} I. M. Isaacs, G. Navarro and T. R. Wolf, Finite group elements where no irreducible character vanishes, J. Algebra, 222(1999) 413--423.

\bibitem{KB} L. Kazarin and Y. Berkovich, On Thompson's Theorem, J. Algebra 220 (1999) 574--590.



%\bibitem{GLM} E. Gianneli, S. Law and  E. Mcdowell, Degrees and prime power orders zeros of characters of symmetric and alternating groups, Bull. London Math. Soc. (2025), 1–21, DOI:10.1112/blms.70160.
%\bibitem{Lew14} M. L. Lewis,
%Classifying Camina groups: A theorem of Dark and Scoppola, Rocky Mountain Journal of Mathematics 44 (2014) 591--597.

%\bibitem{LMPST} M. L. Lewis, L. Morotti, E. Pacifici, L. Sanus, H. P. Tong-Viet, Groups with conjugacy class that is the difference of two normal subgroups: arXive:2603.25407v1. 

%\bibitem{LM} M. L. Lewis, L. Morotti, E. Pacifici, L.  Sanus and H. P. Tong-Viet,   On common zeros of characters of finite groups, Algebr. Represent. Theory, 28 (2025), 395--406.
%https://doi.org/10.1007/s10468-025-10320-1.

%\bibitem{LM} Mark L. Lewis, Lucia Morotti, Emanuele Pacifici, Lucia Sanus, and Hung P. Tong-Viet,
%On common zeros of characters of finite groups,  Algebr. Represent. Theory 28 (2025), no. 2, 395--406.



%\bibitem{Liu-Lu} Y. Liu and Z. Q. Lu, Nonsolvable $D_2$-groups, Acta. Math. Sin. Engl. Ser. 31(11) (2015) 1683–1702.


\bibitem{MNO} G. Malle, G. Navarro, and J. Olsson, Zeros of characters of finite groups,
J. Group Theory 3 (2000) 353--368.


%\bibitem{hung-tong} H. P. Tong-Viet, Finite nonsolvable groups with many distinct character degrees, Pacific J. Math. 268(2) (2014) 477–492.

 



%\bibitem{kohler}  C. K\"{o}hler, H. Pahlings, Regular orbits and the k(GV )-problem, in: Groups and Computation III:
%Proceedings of the International Conference at the Ohio State University, 1999, 2001-209-228.


%\bibitem{Lad} F. Ladisch, Groups with anticentral elements, Comm. Algebra, 36(2008), 2883--2894.

%\bibitem{Lew09} M. L. Lewis, Generalizing Camina groups and their character tables, J. Group Theory, 12 (2009), 209--218.

%\bibitem{Lew} M. L. Lewis, The vanishing-off subgroup, J. Algebra, 321 (2009), 1312--1325.


%\bibitem{moratti-tong-viet} L. Morotti, H. P. Tong-Viet, Proportions of vanishing elements in finite groups, Israel J. Math. 246 (2021), 441–457 DOI: 10.1007/s11856-021-2256-4.

%\bibitem{moreto-tiep} A. Moreto, P. H. Tiep, Nonsolvable groups have a large proportion of vanishing elements, israel. J. Math.  254 (2023), 229–242
%DOI: 10.1007/s11856-022-2395-2.


%\bibitem{Zhang} P. P. Palfy, Groups with two nonlinear irreducible representation, Ann. Uni. Sci. Budapest. Eotvos.Sect. math. 24  (1981) 181-192. 

%\bibitem{hung2016} N. N. Hung, P. H. Tiep, Irreducible characters of even degree
%and normal Sylow 2-subgroups, Math. Proc. Camb. Phil. Soc. (2017), 162, 353–365.
%\bibitem{berkovich1} Y. Berkovich, Generalizations of M-groups, Proc. Amer. Math. Soc. 123 (1995),
%3263–3268.

%\bibitem{berkovich2}  Y. Berkovich, On the Taketa theorem, J. Algebra 182 (1996), 501–510.

%\bibitem{MNO00} G. Malle, G. Navarro and J. B. Olsson, Zeros of characters of finite groups, \emph{J. Group Theory} \textbf{3} (2000) 353--368.
%\bibitem{navarro} G. Navarro, Characters and Blocks of Finite Groups. Cambridge University Press, 1998.

\bibitem{seitz} G. M. Seitz, Finite groups having only one irreducible representation of degree greater than one, Proc. Am. Math. Soc. 19 (1986) 459--461. 




%\bibitem{Schmid} P. Schmid, Extending the Steinberg representation, J. Algebra 150 (1992) 254–256.





%\bibitem{isaacs1} I. M. Isaacs, Generalizations of Taketa’s theorem on the solvability of M-groups,
%Proc. Amer. Math. Soc. 91 (2) (1984), 192–194.

%\bibitem{isa} I.M. Isaacs, M. Loukaki, and A. Moret\'o, The average degree of an irreducible character of a
%finite group, Israel J. Math. 197 (2013), 55–67.



%\bibitem{lmt} T. Le, J. Moori and H. P. Tong-Viet, A generalization of M-group, J. Algebra 374
%(2013), 27–41.



%\bibitem{feit} W. Feit, Groups which have a faithful representation of degree less than p −1, Trans. Amer. Math. Soc. 112 (1964) 287–303.

%\bibitem{Feit} W. Feit, The current situation in the theory of finite simple groups, Actes Congr. Internat. Math. 1 (1970) 55–93.

%\bibitem{f.t}W. Feit, J.G. Thompson, Groups which have a faithful representation of degree less than $(p −1)/2$, Pacific J. Math. 11 (1961) 1257–1262.



%\bibitem{moreto} A. Moret\'o, H. N.  Nguyen, On the average character degree of
%finite groups, Bulletin of the London Mathematical Society, 46 (3) (2014), 454–462.

%\bibitem{phy} H. Pan, N. N. Hung and Y. Yang, On the sum of character degrees coprime to p and p-nilpotency of finite groups, J. Pure and applied algebra, 225(9)  (2021), 1-14.
%\bibitem{qp} G. Qian, Two results related to the solvability of M-groups, J. Algebra 323 (2010), 3134-3141.

%\bibitem{qian} G. Qian, On the average character degree and the average class size in finite groups, J. Algebra 423 (2015), 1191-1212.

%\bibitem{qianpri} G. Qian, Nonsolvable groups with few primitive character degrees, J.  Group Theory  21 (2) (2018), 295-318.



%\bibitem{tiep}P.H. Tiep, A.E. Zalesskii, Minimal characters of finite classical groups, Comm. Algebra 24(6) (1996) 2093–2167
\end{thebibliography}
\end{document}